\documentclass{article}

    \PassOptionsToPackage{compress, round}{natbib}

\usepackage{arxiv}

\usepackage[utf8]{inputenc} 
\usepackage[T1]{fontenc}    
\usepackage{hyperref}       
\usepackage{url}            
\usepackage{booktabs}       
\usepackage{amsfonts}       
\usepackage{nicefrac}       
\usepackage{microtype}      
\usepackage{natbib}

\usepackage{comment}
\usepackage{mathtools}
\usepackage{graphicx}
\usepackage{amsmath,amssymb,amsfonts,amsthm,amsxtra,bm}
\usepackage{amsthm}
\usepackage{xspace}
\usepackage{enumerate}
\usepackage{colortbl}
\usepackage{makecell,multirow}       
\usepackage{thmtools}  
\usepackage{xspace}
\usepackage{footnote, tablefootnote}
\usepackage{float}
\floatstyle{plaintop}
\restylefloat{table}
\usepackage{epstopdf}
\usepackage{latexsym}
\usepackage{algorithm, algorithmicx, algpseudocode}
\usepackage{todonotes}
\usepackage{thmtools, thm-restate}
\usepackage{bbm}
\newtheorem{theorem}{Theorem}
\newtheorem{lemma}[theorem]{Lemma}

\newtheorem{definition}{Definition}

\newcommand{\E}{\mathbb{E}}

\DeclareMathOperator{\supp}{supp}
\DeclareMathOperator*{\argmax}{argmax}

\newcommand{\Ber}{\mathrm{Bernoulli}}

\newcommand{\bigotilde}{\tilde{\mathcal{O}}}

\newcommand{\cO}{\mathcal{O}}

\newcommand{\cC}{\mathcal{C}}

\newcommand{\cF}{\mathcal{F}}

\newcommand{\cP}{\mathsf{P}}

\newcommand{\cX}{\mathcal{X}}
\newcommand{\cY}{\mathcal{Y}}

\newcommand{\norm}[1]{\left\|#1\right\|}
\newcommand{\abs}[1]{\left\vert#1\right\vert}

\newcommand{\pr}{\mathbb{P}}
\newcommand{\R}{\mathbb{R}}

\newcommand{\vb}{\boldsymbol{b}}

\newcommand{\vg}{\boldsymbol{g}}

\newcommand{\ve}{\boldsymbol{e}}
\newcommand{\vv}{\boldsymbol{v}}

\newcommand{\vu}{\boldsymbol{u}}

\newcommand{\vx}{\boldsymbol{x}}
\newcommand{\vy}{\boldsymbol{y}}
\newcommand{\vby}{\bar{\boldsymbol{y}}}
\newcommand{\vz}{\boldsymbol{z}}
\newcommand{\half}{\frac{1}{2}}

\DeclareMathAlphabet{\mathsfit}{T1}{\sfdefault}{\mddefault}{\sldefault}
\SetMathAlphabet{\mathsfit}{bold}{T1}{\sfdefault}{\bfdefault}{\sldefault}

\makeatletter
\newcommand{\vast}{\bBigg@{3}}
\newcommand{\Vast}{\bBigg@{3.5}}
\makeatother

\newcommand{\ky}{\kappa}
\newcommand{\sigmav}{\sigma}
\newcommand{\sigmas}{\lambda}

\newcommand{\0}{\mathbf{0}}
\newcommand{\fsc}{f^{\text{sc}}}
\newcommand{\Asc}{A^{\text{sc}}}

\newcommand{\fnc}{\bar{f}^{\text{nc}}}
\newcommand{\ft}{\bar{f}^{\text{nc-sc}}}
\newcommand{\fmt}{\bar{f}_m^{\text{nc-sc}}}

\newcommand{\fft}{{f}^{\text{nc-sc}}}
\newcommand{\ffmt}{{f}_m^{\text{nc-sc}}}

\newcommand{\ffts}{{f}^{\text{nc-sc-sg}}}
\newcommand{\ffmts}{{f}_m^{\text{nc-sc-sg}}}
\newcommand{\fts}{\bar{f}^{\text{nc-sc-sg}}}
\newcommand{\fmts}{\bar{f}_m^{\text{nc-sc-sg}}}
\newcommand{\gs}{{h}^{\text{sg}}}
\newcommand{\gd}{{h}}

\newcommand{\bcX}{\cC^{T}_{R_1}}
\newcommand{\bcY}{\cC^{n(T-1)}_{nR_2}}
\newcommand{\bcZ}{\cC^{T-1}_{R_1}}

\begin{document}

\title{Complexity Lower Bounds for Nonconvex-Strongly-Concave Min-Max Optimization}

\author{\name Haochuan Li\footnotemark[1] \email{haochuan@mit.edu} \\
\name Yi Tian\footnotemark[1] \email{yitian@mit.edu} \\
\name Jingzhao Zhang\footnotemark[1] \email{jzhzhang@mit.edu}\\
\name Ali Jadbabaie\footnotemark[1] \email{jadbabai@mit.edu} \\
\addr{Massachusetts Institute of Technology}
}

\renewcommand*{\thefootnote}{\fnsymbol{footnote}}
\footnotetext[1]{Supported in parts by ONR grant N00014-20-1-2394, an MIT-IBM Watson grant and NSF BIGDATA grant 1741341.}
\renewcommand*{\thefootnote}{\arabic{footnote}}

\maketitle

\begin{abstract}
    We provide a first-order oracle complexity lower bound for finding stationary points of min-max optimization problems where the objective function is smooth, nonconvex in the minimization variable, and strongly concave in the maximization variable. We establish a lower bound of $\Omega\left(\sqrt{\ky}\epsilon^{-2}\right)$ for deterministic oracles, where $\epsilon$ defines the level of approximate stationarity and $\kappa$ is the condition number. Our analysis shows that the upper bound achieved in~\citep{lin2020nearoptimal} is optimal in the $\epsilon$ and $\kappa$ dependence up to logarithmic factors. For stochastic oracles, we provide a lower bound of $\Omega\left(\sqrt{\ky}\epsilon^{-2} + \ky^{1/3}\epsilon^{-4}\right)$. It suggests that there is a significant gap between the upper bound $\cO(\kappa^3 \epsilon^{-4})$ in \citep{lin2020gradient} and our lower bound in the condition number dependence. 
\end{abstract}

\section{Introduction}

In this paper, we study the oracle complexity lower bound of the following min-max optimization problem: 
\begin{align}\label{eq:min-max}
	\min_{\vx\in\cX}\max_{\vy\in\cY}f(\vx;\vy),
\end{align}
where $\cX\subset \R^m$ and $\cY\subset \R^n$ are nonempty closed convex sets. 
Such a problem arises in a wide range of applications, e.g., two-player zero-sum games~\citep{neumann1928theorie}, Generative Adversarial Networks~\citep{goodfellow2014generative}, robust optimization~\citep{ben2009robust}, including defending adversarial attacks~\citep{goodfellow2014explaining,madry2017towards}.

The research in min-max problems~\eqref{eq:min-max} has a long history. If $f$ is linear in both $\vx$ and $\vy$, the problem is known as bilinear min-max optimization, for which von Neumann's min-max theorem~\citep{neumann1928theorie} guarantees the existence of a saddle point, a point $(\vx^\ast, \vy^\ast)$ that satisfies strong duality
\begin{align}\label{eq:strong-dual}
    f(\vx^\ast;\vy^\ast)=\min_{\vx\in \cX} \max_{\vy\in \cY} f(\vx;\vy) = \max_{\vy\in \cY} \min_{\vx\in \cX} f(\vx;\vy).
\end{align}
Later, \citet{sion1958general} generalized the statement to convex-concave functions under certain regularity conditions, which led to the theory of KKT conditions and Lagrangian duality. Aside from studying the existence of strong duality, \citet{nemirovski2004prox, nesterov2005smooth} proposed algorithms that can find approximate saddle points for convex-concave objectives with an $\cO(1/\epsilon)$ convergence rate. The rate was proven to be worst-case optimal even for min-max problems with a bilinear cross-term~\citep{ouyang2019lower}. 

Going beyond the assumption of convex-concavity poses a big challenge. In general nonconvex-nonconcave min-max optimization problems, a saddle-point may not exist \citep{jin2020local}. Defining a notion of min-max points that is simultaneously nontrivial and tractable is still an open question. If one uses the standard definition of saddle points as in~\eqref{eq:strong-dual}, then determining its existence is known to be NP-hard~\citep{daskalakis2020complexity}, and finding an approximate local saddle point is PPAD-complete~\citep{daskalakis2020complexity}. Alternatively, a relaxed definition of min-max point known as Stackelberg equilibrium is guaranteed to exist~\citep{jin2020local}, yet only local asymptotic convergence is known~\citep{jin2020local,wang2019solving}.

Due to the lack of a well-formulated suboptimality measure for general nonconvex-nonconcave problems, recent research has considered problems with special structures. For example, ~\citet{lin2020gradient, lin2020nearoptimal} studied nonconvex-concave problems, \citet{yang2020global} studied problems satisfying the two-sided Polyak-Lojasiewicz inequality, \citet{diakonikolas2021efficient} studied problems with weak Minty variational inequality solutions, \citet{mangoubi2020greedy} considered an algorithm specific equilibrium, and \citet{lee2020linear, daskalakis2021independent, wei2021last} studied two-player zero-sum stochastic games. 

Our work focuses on the analysis of gradient oracle complexity in the nonconvex-strongly-concave min-max setting. In the standard convex-concave setup, \citet{lin2020nearoptimal} designed a set of algorithms that achieve near-optimal gradient complexity for all the three variants of convex-concave min-max optimization. However, in the nonconvex-concave setting, it is unclear whether their algorithm is optimal given that no matching lower bound exists. In this work, our main contribution is as follows:

\begin{itemize}
    \item We construct an explicit example on which no first-order zero-respecting algorithm can achieve less than $\Omega(\sqrt{\ky} \epsilon^{-2})$ oracle complexity for smooth nonconvex-strongly-concave problems, matching the $\bigotilde(\sqrt{\ky} \epsilon^{-2})$ complexity in~\citep{lin2020nearoptimal}.
    \item We further extend our result to the stochastic(-oracle) setting, achieving the complexity lower bound $\Omega(\sqrt{\ky}\epsilon^{-2}+\ky^{1/3}\epsilon^{-4})$. Compared against the result in~\citep{lin2020gradient}, our result suggests that the dependency on the condition number may not be optimal. 
\end{itemize}

\subsection{Related work}
\label{sec:related}

\paragraph{Upper bounds.} 
Since standard gradient descent-ascent diverges for even the simplest case of bilinear zero-sum games, researchers proposed different variations to achieve convergence in the convex-concave setting, with possibly additional structures. One of the variations is known as extra-gradient method~\citep{korpelevich1976extragradient, tseng1995linear, nemirovski2004prox, chambolle2011first, yadav2017stabilizing}, which takes two gradient steps per iteration. Another line studies the optimistic gradient method~\citep{rakhlin2013optimization, daskalakis2017training, gidel2018variational, mertikopoulos2018optimistic, hsieh2019convergence}. \citet{mokhtari2020convergence} proposed a unified analysis through the lens of proximal update. Some works~\citep{alkousa2019accelerated,lin2020nearoptimal} focused on accelerating the known rates in terms of the conditional number dependence. \citet{golowich2020last} studied last iterate convergence. 

Convergence under the nonconvex-concave setup is less studied. \citet{rafique2018non} studied a proximal version of gradient methods. \citet{lin2020gradient} studied gradient descent ascent with different time scales. \citet{thekumparampil2019efficient} proposed an implicit algorithm for the nonconvex-concave setting. Many other works also studied different variations of convergence~\citep{nouiehed2019solving, kong2019accelerated, lu2020hybrid, ostrovskii2020efficient}.

\paragraph{Lower bounds.} 
Our work follows the complexity framework introduced by Nemirovski in the 1980s. The most classical lower bounds are due to Nemirovski and Nesterov, and are discussed in their textbooks~\citep{Nemirovsky1992InformationbasedCO, nesterov2018lectures}. The central idea is known as zero-chains, for which all first-order zero-respecting methods activate the coordinates one by one.
Recently, much progress was made~\citep{fang2018spider,carmon2019lower,carmon2021lower,arjevani2019lower} by extending the original analysis to the nonconvex setting. Our work builds upon these results. 
Another line of works studies lower bounds for quadratic problems in min-max setting and utilizes a different framework sometimes known as Stationary Canonical Linear Iterative (SCLI)~\citep{arjevani2016iteration, ouyang2019lower, zhang2020lower, ibrahim2020linear}. It exploits the closed-form updates for linear dynamics and studies convergence properties via the transition matrices.

As we were preparing the final draft of this paper, we note that 2-3 weeks ago \citet{zhang2021complexity} provided a similar lower complexity bound in the deterministic nonconvex-strongly-concave setting using a different construction, yet following the same formalism as Nemirovski and Carmon~\citep{nemirovski1983book, carmon2019lower}. Whereas \citet{zhang2021complexity} provided novel analyses with extensions to the upper and lower bounds in the finite-sum setting, our proofs and construction are different and we also extend the results to the stochastic setting.

\section{Preliminaries}
Before introducing our main results, we describe the problem setup and algorithm complexity in this section. First, we introduce the notation. Next, we define the function class and algorithm class. Finally, we provide the formal definition of optimality measure and gradient complexity.

\paragraph{Notation.}
We use bold lower-case letters to denote vectors and use $x_i$ to denote the $i$-th coordinate of vector $\vx$. Let $\supp(\vx):=\left\{i:x_i\neq 0\right\}$ denote the support of $\vx$. 
Let $\norm{\vx}_2$ and $\norm{\vx}_\infty$ denote its $\ell_2$ and $\ell_\infty$ norm respectively. 
For a matrix $A$, we use $A_{i,j}$ to denote its $(i,j)$-th entry and $A^\top$ its transpose. We use $\norm{A}_2$ to denote its spectral norm and $\det(A)$ its determinant.
We use calligraphic upper-case letters to denote sets, as in $\cX$ and $\cY$. 
For a nonempty closed convex set $\cX\in\R^d$, let $\cP_{\cX}$ denote the Euclidean projection onto $\cX$. 
We also use a semi-comma besides commas to split minimization and maximization variables of a function. For example, when we write $f(\vx,\vz;\vy)$, it means $\vx$ and $\vz$ are the variables to minimize and $\vy$ is the variable to maximize.
Finally, we use the standard $\cO(\cdot)$ and $\Omega(\cdot)$ notation, with $\bigotilde(\cdot)$ and $\Tilde{\Omega}(\cdot)$ further hiding log factors.

\subsection{Function class}
\label{sec:def_fun_class}

Our analysis focuses on the class of smooth nonconvex-strongly-concave functions defined below. 

\begin{definition}
	\label{def:function_class}
    Given $L\ge \mu>0$ and $\Delta>0$, we use $\cF(L, \mu, \Delta)$ to denote the set of all functions $f:\cX\times\cY\to \R$ for some nonempty closed convex sets $\cX\subset \R^m$ and $\cY\subset\R^n$ where $m,n\in\mathbb{N}$, which satisfies the following assumptions: 
    \begin{enumerate}
		\item  $f$ is $L$-smooth, that is, for every $\vx_1, \vx_2\in \cX$ and $\vy_1, \vy_2 \in \cY$, 
		\begin{align*}
			\norm{\nabla f(\vx_1; \vy_1) - \nabla f(\vx_2; \vy_2) }_2 \le L \norm{(\vx_1, \vy_1) - (\vx_2, \vy_2)}_2 ; 
		\end{align*}
		\item For any fixed $\vx\in \cX$, $f$ is $\mu$-strongly concave in $\vy$, that is, for any $\vy_1, \vy_2 \in \cY$, 
		\begin{align*}
			f(\vx; \vy_1) \le f(\vx; \vy_2) + \nabla_{\vy} f(\vx; \vy_2) \cdot (\vy_1 - \vy_2) - \frac{1}{2} \mu \norm{\vy_1 - \vy_2}_2^2 ; 
		\end{align*}
        \item $f_m(\0)-\min_{\vx\in\cX} f_m(\vx)\le \Delta$, where $f_m(\vx) := \max_{\vy\in\cY}f(\vx;\vy)$.
    \end{enumerate}
\end{definition}

For any fixed $\vy\in \cY$, $f$ is potentially nonconvex in $\vx$. Note that $\cF(L, \mu, \Delta)$ includes functions with domain on $\R^m\times \R^n$ for all $m,n\in\mathbb{N}$, following the framework of dimension-free convergence guarantees \citep{nemirovski1983book, nesterov2018lectures}. In Section~\ref{sec:our_construction}, we will construct a hard instance with domain dimensions $m, n$ both growing inversely in the required accuracy $\epsilon$.

\subsection{Algorithm class}

In this section, we describe the algorithms of interest for the min-max problems with the function class defined above. In particular, we restrict our analysis to first-order algorithms that optimize objectives using first-order oracles defined below.

\begin{definition}[Deterministic first-order oracle]
	\label{def:deterministic_oracle}
	The deterministic first-order oracle of a differentiable function $f:\cX\to\R$ is a mapping $O: \vx \mapsto \left(f(\vx),\nabla f(\vx)\right)$ for $\vx \in \cX$.
\end{definition}

\begin{definition}[Stochastic first-order oracle]
	\label{def:stochastic_oracle}
	A stochastic first-order oracle with bounded variance $\sigma^2$ of a differentiable function $f:\cX\to\R$ is a mapping $O: \vx \mapsto \left(f(\vx),\vg(\vx,\xi)\right)$ for $\vx\in\cX$, where $\xi$ is a random variable satisfying $\E[\vg(\vx,\xi)]=\nabla f(\vx)$ and $\E\left[\norm{\vg(\vx,\xi)-\nabla f(\vx)}_2^2\right]\le \sigma^2.$
\end{definition}

We say $O$ is a first-order oracle if it is a deterministic or stochastic first-order oracle. 

Furthermore, we consider first-order algorithms satisfying the zero-respecting assumption. Formally, we define first-order algorithms as follows.

\begin{definition}[First-order algorithm]  \label{def:first_order_alg}
    A first-order (zero-respecting) algorithm is one that for any function $f:\cX\to\R$ and its associated first-order oracle $O_f: \vx \mapsto (f(\vx), \vg)$, the $(t+1)$-th iterate $\vx^{t+1}$ satisfies
    \begin{align*}
        \vx^{t+1}\in \left\{\cP_{\cX}(\vv):\supp(\vv)\subset\bigcup_{0\le i\le t}\left(\supp(\vx^i)\cup \supp(\vg^i)\right)\right\}.
    \end{align*}
\end{definition}
Definition~\ref{def:first_order_alg} extends the standard zero-respecting algorithms \citep{carmon2019lower,arjevani2019lower} to the constrained setting. It covers most existing first-order methods used in the literature including (projected) stochastic gradient descent, adaptive methods, and more importantly, the algorithms used in \citep{lin2020gradient,lin2020nearoptimal} which achieve the upper bounds for nonconvex-strongly-concave min-max optimization in the deterministic and stochastic settings, respectively. There is a standard reduction from a lower bound for zero-respecting algorithms to that for arbitrary deterministic algorithms with deterministic or even stochastic first-order oracles~\citep{carmon2019lower,arjevani2019lower}. We defer this extension to future work.

\subsection{Optimality via approximate stationarity}

We measure the progress of solving the nonconvex-strongly-concave problem via the gradient norm of the maximized function with respect to the minimization variable $\vx$.
Following~\cite{ghadimi2013minibatch}, we define the notion of $\epsilon$-stationary points in presence of constraints as follows.

\begin{definition} \label{def:stationary}
    Let $\cX$ be a nonempty closed convex set. A point $\vx\in\cX$ is said to be an $\epsilon$-stationary point of an $L$-smooth function $f:\cX\to\R$ if
    \begin{align*}
        L\norm{\cP_\cX \left[\vx-(1/L)\nabla f(\vx)\right]-\vx}_2\le \epsilon.
    \end{align*}
\end{definition}

Note that the algorithm studied in~\citep{lin2020gradient} also assumes bounded domain and used the same notion of stationarity. This definition of stationary points reduces to $\norm{\nabla f(\vx)}_2\le \epsilon$ when $\cX=\R^d$ for some $d\in\mathbb{N}$.
For any $f\in\cF(L, \mu, \Delta)$,  Definition~\ref{def:stationary} applies to $f_m$ as it is $L_m$-smooth, where $L_m\le (\ky+1)L$ by~\citep[Lemma~4.3]{lin2020gradient}. 

Our goal of solving the nonconvex-strongly-concave min-max optimization problem is to find an $\epsilon$-stationary point of $f_m$. We show that no deterministic first-order algorithm can achieve less than $\Omega(\sqrt{\ky} \epsilon^{-2})$ gradient complexity for smooth nonconvex-strongly-concave problems, matching the $\bigotilde(\sqrt{\ky} \epsilon^{-2})$ complexity in \citep{lin2020nearoptimal}. We further extend our result to the stochastic setting, achieving a complexity lower bound $\Omega(\sqrt{\ky}\epsilon^{-2}+\ky^{1/3}\epsilon^{-4})$.

\section{Main results}

In this section, we present our main results on the minimum number of gradient oracle calls required to find an $\epsilon$-stationary point. The results for deterministic and stochastic settings are presented in the following two subsections respectively.

\subsection{Lower bound on first-order oracle complexity in the deterministic setting}

Nonconvex-strongly-concave min-max optimization subsumes nonconvex optimization, the lower bound $\Omega(L\Delta /\epsilon^2)$ in nonconvex optimization~\citep{carmon2019lower} also holds for nonconvex-strongly-concave min-max optimization. However, compared with the $\bigotilde(L \Delta \sqrt{\kappa} / \epsilon^2)$ upper bound~\citep{lin2020nearoptimal}, a $\sqrt{\kappa}$ factor is missing. Our main result below fills this gap, showing that the known rate by~\cite{lin2020gradient} is optimal up to log factors.

\begin{theorem} \label{thm:lb-nsc}
	For any $\mu, L, \Delta, \epsilon>0$ such that $\ky=L/\mu\ge 1$, there exists a function instance $f:\R^m\times\R^n\to \R$ in $\cF(L, \mu, \Delta)$ for some $m,n\in\mathbb{N}$ with its deterministic first-order oracle such that for any first-order algorithm, we have $ \norm{ \nabla f_m(\vx^t)}_2> \epsilon$, where $f_m(\vx):=\max_{\vy}f(\vx,\vy)$, whenever
	\begin{align*}
		t\le  \frac{c_0 L \Delta\sqrt{\ky}}{\epsilon^2},
	\end{align*}
	where $c_0$ is a numerical constant.
\end{theorem}
Similar to the lower bound in~\citep{carmon2019lower}, our lower bound applies to dimension-free optimization in nature. That is, for a given $\epsilon$, we construct a hard instance with dimension $d = \cO(1/\epsilon^2)$, which can be very large if $\epsilon$ is small. The discussion on the proof is deferred until Section~\ref{sec:our_construction}.

\subsection{Lower bound in the stochastic setting}
In the stochastic setting, \citet{arjevani2019lower} provided a lower bound of $\Omega\left(L\Delta\sigma^2/\epsilon^4\right)$ for smooth nonconvex optimization which is a special case of nonconvex-strongly-concave min-max optimization. Our analysis improves this bound by a factor of $\ky^{1/3}$ in Theorem~\ref{thm:lb-nsc-s}.

\begin{theorem} 
	\label{thm:lb-nsc-s}
	For any $\mu, L, \Delta, \epsilon, \sigmav>0$ such that $\ky=L/\mu\ge 1$, there exists a function instance $f:\cX\times\cY\to \R$ in $\cF(L, \mu, \Delta)$ and a stochastic first-order oracle $O$ for $f$ with variance $\sigmav^2$ such that for any first-order algorithm, we have $\E\left[L_m\norm{\cP_\cX \left[\vx^t-(1/L_m)\nabla f_m(\vx^t)\right]-\vx^t}_2\right]> \epsilon$, where $f_m(\vx):=\max_{\vy}f(\vx,\vy)$, whenever
		\begin{align*}
		t\le  c_0L\Delta\left(\frac{\sqrt{\ky}}{\epsilon^2}+\frac{\ky^{1/3}\sigma^2}{\epsilon^4}\right)
	\end{align*}
	where $L_m$ is the smoothness parameter of $f_m$ and $c_0$ is a numerical constant.
\end{theorem}

\citet{lin2020gradient} provided an upper bound of $\cO\left(\ky^3\epsilon^{-4}\right)$. 
Therefore, there is a gap between our lower bound and their upper bound in terms of the dependency on $\ky$. We defer closing this gap to future work. 

In summary, our proposed lower bounds improve the known ones by a multiplicative dependence on the condition number. Before proceeding to discuss the concrete techniques, we first summarize the general framework for establishing lower bounds in the next section.

\section{Framework for proving lower bound}

In this section, we provide an outline for the proof of the lower bound in \citep{nesterov2018lectures} and \citep{carmon2019lower}, which lay the foundation for our construction of the hard instance. Both works utilize the notion of zero-chains \citep{carmon2019lower}, which instantiates a class of hard functions for optimization. We first define a zero-chain and then discuss how it is used to establish complexity lower bounds.

\begin{definition}[Zero-chain]
\label{def:zero-chain}
	We say a function $f:\cX\to \R$, where $\cX\subset \R^d$, is a (first-order) zero-chain if for every $1\le i\le d$,
	\begin{align*}
		\supp(\vx) := \{i:x_i\neq 0\} \subset \{1, \ldots, i-1\} \implies \supp(\nabla f(\vx))\subset \{1, \ldots, i\}.
	\end{align*}
\end{definition}

Suppose the domain $\cX\subset \R^d$ satisfies $\supp(\cP_{\cX}(\vx))=\supp(\vx)$ for all $\vx\in\R^d$, i.e., projecting $\vx$ onto $\cX$ does not change its support. For example, this requirement holds when $\cX$ is a hypercube or the whole space $\R^d$. 
If we run a first-order algorithm on a zero-chain initialized at $\vx^0 = \0$ (which we assume to hold without loss of generality) with a deterministic first-order oracle, then at each iteration, at most one new coordinate of $\vx$ becomes nonzero (``discovered''). Therefore, $\supp(\vx^t)\subset \{1,\ldots,t\}$. Then we obtain a lower complexity bound of $T$ suppose we can show a good solution exists only if at least $T$ coordinates are discovered. 

Therefore, the key to proving a lower bound of first-order algorithms with a deterministic oracle is to find a function $f$ such that: 
\begin{enumerate}
	\item it is a zero-chain that belongs to the function class we are interested in; and that
	\item we cannot obtain an $\epsilon$-optimal solution if the $t$-th coordinate of $\vx$ is zero for every $t\ge T$.
\end{enumerate}
This is actually a general strategy for proving lower bounds of first-order methods, used in the lower bound construction both by~\citet{carmon2019lower} for smooth nonconvex optimization and by us here for nonconvex-strongly-concave min-max optimization.

In the stochastic setting, we utilize the generalized notion known as the probability-$p$ zero-chain~\citep{arjevani2019lower} to prove a lower bound.

\begin{definition}[Probability-$p$ zero-chain]
	A function $f:\cX\to\R$ with a stochastic first-order oracle $O: \vx \mapsto(f(\vx),\vg(\vx,\xi))$ is a probability-$p$ zero-chain if
	\begin{align*}
		\supp(\vx) \subset \{1, \ldots, i - 1\} \quad\implies\quad\left\{\begin{array}{ll}
		     & \pr\left(\supp(\vg(\vx,\xi)) \not \subset \{1, \ldots, i - 1\}\right) \le p, \\
		     & \pr\left(\supp(\vg(\vx,\xi)) \subset \{1, \ldots, i \}\right) =1.
		\end{array} \right.
	\end{align*}
\end{definition}

For a probability-$p$ zero-chain, at each iteration, a new coordinate is discovered with probability at most $p$ if $\supp(\cP_{\cX}(\vx))=\supp(\vx)$ for all $\vx$. 
Therefore, it takes at least $1/p$ steps in expectation to discover a new coordinate. Formally, the following lemma states that it takes $\cO(T/p)$ iterations to reach the end of a probability-$p$ zero-chain with length $T$.

\begin{lemma}[{\citep[Lemma 1]{arjevani2019lower}}]
	\label{lem:arj-l1}
	Let $f:\cX\to\R$, where $\cX\in\R^T$ satisfying $\supp(\cP_{\cX}(\vx))=\supp(\vx)$ for all $\vx\in\R^T$, be a probability-$p$ zero-chain with a stochastic first-order oracle. For any first-order algorithm, we have with probability at least $1-\delta$, $x_T^t=0$ for all $t \le \frac{T - \log(1/\delta)}{2p}$.
\end{lemma} 
Therefore, the gradient complexity is enlarged by a factor of $1/p$ compared with the deterministic setting. To obtain a lower bound in the stochastic setting, we first find a zero-chain $f$ satisfying the two requirements in the general strategy we presented above. Then we construct a stochastic first-order oracle which discovers the next coordinate with probability $p$ so that we obtain a probability-$p$ zero-chain. Note that we can not choose an arbitrarily small $p$ since we need to ensure the variance of the stochastic oracle is bounded.

For the ease of exposition of our construction, we now briefly review the lower bound construction for smooth strongly-convex minimization by~\citet{nesterov2018lectures} and that for smooth nonconvex optimization by~\citet{carmon2019lower}. As we shall see in the next section, the hard functions in these two cases are the building blocks of our construction.

\subsection{Smooth strongly-convex minimization}

\citet{nesterov2018lectures} constructed the following hard instance for smooth, strongly-convex functions: 
\begin{align}
	\fsc(\vx) 
	:= \frac{\mu(\kappa-1)}{8}\left((x_1-1)^2+\sum_{i=1}^\infty (x_i-x_{i+1})^2\right) + \frac{\mu}{2} \norm{\vx}_2^2.
	\label{equ:nesterov}
\end{align}
Equivalently, $\fsc(\vx) = \frac{\mu(\kappa-1)}{8} (\vx^{\top} \Asc \vx -2x_1+ 1) + \frac{\mu}{2} \norm{\vx}_2^2$ for the tri-diagonal matrix $\Asc$ given by 
\begin{align} \label{equ:def_A_nest}
	\Asc := 
	\begin{pmatrix}
		2&-1&  \\
		-1&2&-1\\
		&-1&2&\ddots\\
		&&\ddots&\ddots&\ddots
	\end{pmatrix}.
\end{align}
It is straightforward to verify that $\Asc$ is positive semi-definite and $\norm{\Asc}_2\le 4$. Hence, $\fsc$ is $L$-smooth and $\mu$-strongly-convex for $L:=\mu\kappa$. Importantly, if $\supp(\vx)\subset \{1,\ldots,i-1\}$, we can verify:
\begin{align*}
    \supp\left(\nabla \fsc(\vx)\right)=\supp\left( \Asc\vx \right)\cup \supp\left( \vx \right)\cup\{1\}\subset\{1,\ldots,i\},
\end{align*}
where we use the fact that $\supp\left( \Asc\vx \right)\subset\{1,\ldots,i\}$ because $\Asc$ is a tri-diagonal matrix. Hence, by Definition~\ref{def:zero-chain}, $\fsc$ is a zero-chain. Based on the general strategy above, it suffices to lower bound $\fsc(\vx)-\fsc(\vx^\ast)$ when fewer than $T$ coordinates are non-zero. Note that the minimizer $\vx^{\ast}$ of $\fsc$ is given by $x^\ast_i=q^i$ where $q = \frac{\sqrt{\kappa}-1}{\sqrt{\kappa}+1}$. Then with fewer than $T$ gradient oracles, for every $t\ge T$, the $t$-th coordinate of $\vx$ is still zero. Therefore, 
\begin{align*}
	\fsc(\vx)-\fsc(\vx^\ast)\ge \frac{\mu}{2}\norm{\vx-\vx^\ast}_2^2\ge\frac{\mu}{2}\sum_{i=T}^\infty (x_i^\ast)^2\ge \frac{\mu}{2}q^{2T}\norm{\vx^0-\vx^\ast}_2^2.
\end{align*}
Hence, to find a solution satisfying $\fsc(\vx)-\fsc(\vx^\ast)\le\epsilon$, we need gradient complexity $T\ge \widetilde{\Omega}(\sqrt{\kappa})$.

\subsection{Smooth nonconvex minimization}

In order to prove the lower bound for smooth nonconvex minimization, \citet{carmon2019lower} constructed the following unscaled function $\fnc:\R^T\to\R$:
\begin{align}
    \fnc(\vx) := -\Psi(1)\Phi(x_1)+\sum_{i=2}^T \left[\Psi(-x_{i-1})\Phi(-x_i)-\Psi(x_{i-1})\Phi(x_i)\right],
    \label{equ:carmon}
\end{align}
where the component functions are
\begin{align*}
    \Psi(x) := \left\{
		\begin{array}{lr}
			0 & x\le 1/2\\
			\exp{\left(1-\frac{1}{(2x-1)^2}\right)} & x>1/2
		\end{array}
		\right. \quad \text{and} \quad 
	\Phi(x) := \sqrt{e}\int_{-\infty}^x e^{-\half t^2}dt.
\end{align*}
 We enumerate all relevant properties of $\Phi$ and $\Psi$ used in the analysis in the following lemma.

\begin{lemma}[{\citep[Lemma 1]{carmon2019lower}}] The functions $\Phi$ and $\Psi$ satisfy
\begin{enumerate}[i.]
    \item For all $x\le 1/2$ and $k\in\mathbb{N}$, $\Psi^{(k)}(x)=0$. 
    \item For all $x\ge 1$ and $|y|<1$, $\Psi(x)\Phi'(y)>1$.
    \item Both $\Psi$ and $\Phi$ are infinitely differentiable. For all $k\in\mathbb{N}$, we have
    \begin{align*}
        \sup_{x}|\Psi^{(k)}(x)|\le \exp{\left(\frac{5k}{2}\log(4k)\right)} \quad\text{ and }\quad \sup_{x}|\Phi^{(k)}(x)|\le \exp{\left(\frac{3k}{2}\log\frac{3k}{2}\right)}.
    \end{align*}
    \item The functions and derivatives $\Psi$, $\Psi'$, $\Phi$, $\Phi'$ are non-negative and bounded, with
    \begin{align*}
        0<\Psi<e,0<\Psi'<\sqrt{54/e},0<\Phi<\sqrt{2\pi e},0<\Phi'<\sqrt{e}.
    \end{align*}
\end{enumerate}
    \label{lem:properties_phi_psi}
\end{lemma}
Note that $\Psi(0)=\Psi'(0)=0$ by Lemma~\ref{lem:properties_phi_psi}.\textit{i}. Then it is  to verify that $\frac{\partial \fnc(\vx)}{\partial x_i}=0$ if $x_i=x_{i-1}=0$. Therefore, if $\supp(\vx)\subset\{1,\ldots,i-1\}$, i.e., $x_j=0$ for all $j\ge i$, we have $\frac{\partial \fnc(\vx)}{\partial x_j}=0$ for all $j\ge i+1$. Hence, $\supp(\nabla \fnc)\subset \{1,\ldots ,i\}$, which implies $\fnc$ is a zero-chain. Define $x_0\equiv1$ for simplicity. As long as the algorithm has not reached the end of the chain, there must be a phase transition point $1\le k\le T$ such that $|x_k|<1$ and $|x_{k-1}|\ge 1$. Using Lemma~\ref{lem:properties_phi_psi}.\textit{ii}, one can bound $\norm{\nabla \fnc(\vx^t)}_2\ge \abs{\frac{\partial \fnc(\vx)}{\partial x_k}}>1$.
Following the general strategy discussed above, by appropriately rescaling $\fnc$ so that it meets the requirement of the function class of interest, \citet{carmon2019lower} derived a lower bound of $T_{\text{nc}}:= \Omega\left(1/\epsilon^2\right)$ gradient oracles.

\section{Our construction}
\label{sec:our_construction}

The hard examples discussed in the previous section show that one can easily construct an additive lower bound by making the objective as the sum of two functions $\fsc$ and $\fnc$. However, if we would like to improve the lower bound by a multiplicative factor in the condition number, it is far from obvious how one should compose $\fsc$ with $\fnc$. We describe our approach in the subsections below.

\subsection{Construction of the hard instance in the deterministic setting}
\label{sec:deterministic_construction}
Before discussing the details of our construction to prove Theorem~\ref{thm:lb-nsc}, we first highlight the main difficulty and our approach at a high level to provide insight and intuition. First, we need to pinpoint the main difficulty.

\paragraph{The main difficulty.}
Given the aforementioned lower bound strategy, we need to find a zero-chain $\ft\in\cF(L, \mu, \Delta)$ as our hard instance. A natural way to construct $\ft$ is to combine the ideas from the hard instances in \citep{nesterov2018lectures} and \citep{carmon2019lower}. The main challenge is how to find a good way of combination such that the two components do not interfere with each other's essential properties and that their strengths can be exploited multiplicatively to contribute to the lower bound. 

\begin{figure}[t]
    \centering
    \includegraphics[trim= 0cm 6cm 0cm 0cm, clip, width=\textwidth]{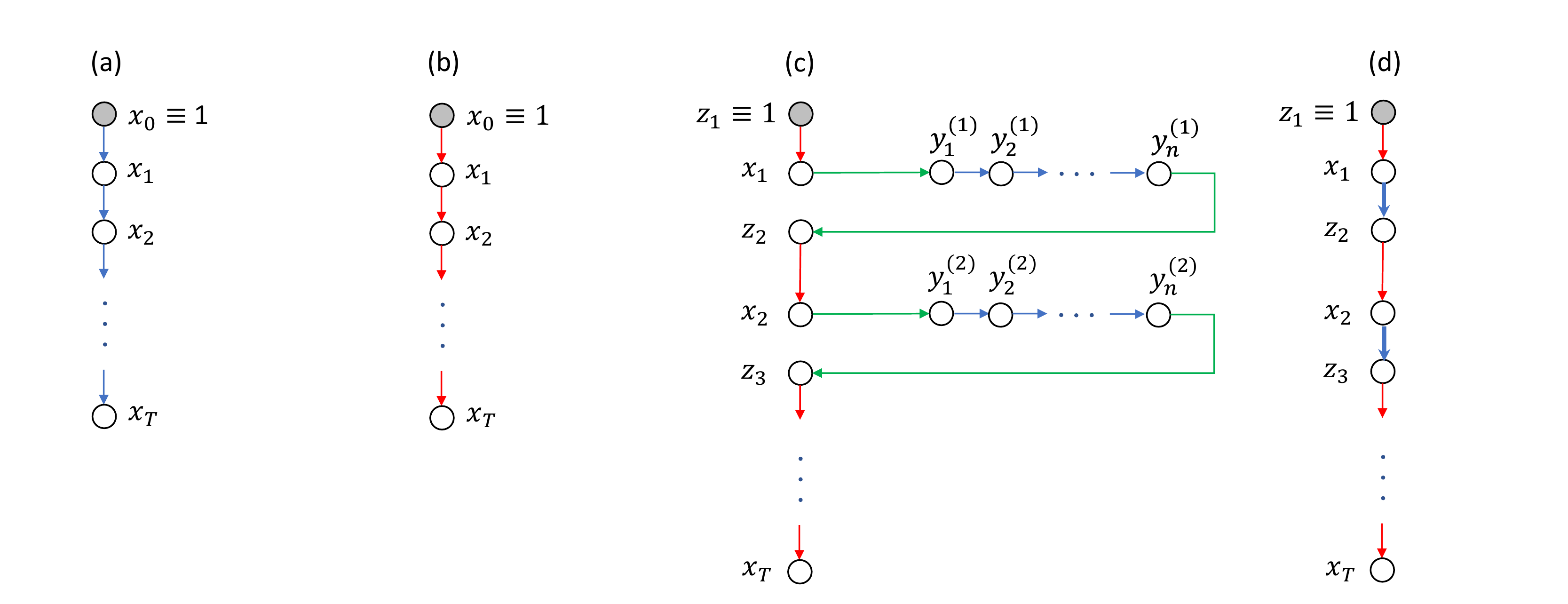}
    \caption{Chains (a, b, c, d) correspond to $\fsc,\fnc,\ft,\fmt$, respectively.  Arrows in different colors represent different types of connections between two variables. For example, a blue arrow from $u$ to $v$ corresponds to a term with $(u-v)^2$ in the function. A red arrow corresponds to $\Psi(-u)\Phi(-v)-\Psi(u)\Phi(v)$. Green and bold blue arrows are quadratic connections that need to be carefully designed. The effective length of a bold blue arrow is $n$ and that of all other arrows is $1$. }
    \label{fig:chain_plot}
\end{figure}

\paragraph{ The key idea.}
The novel combination structure we propose is illustrated in Figure~\ref{fig:chain_plot}. The body of the chain is like $\fnc$~\eqref{equ:carmon}, which contains the minimization variables $\vx, \vz$. Then we insert $T-1$ $\fsc$-style~\eqref{equ:nesterov} sub-chains with length $n$, which contain the maximization variable $\vy$, into the main chain in the way shown in Figure~\ref{fig:chain_plot}(c). Figure~\ref{fig:chain_plot}(d) shows the function $\fmt(\cdot) := \max_{\vy}\ft(\cdot;\vy)$ obtained by maximizing $\vy$. Note that the effictive length of chain (d) is $\cO(nT)=\cO(\sqrt{\ky}T)$, for the largest allowed choice of $n = O(\sqrt{\ky})$. Therefore, suppose we can show that $\fmt$ shares similar properties to $\fnc$ used in \citep{carmon2019lower}, we obtain a lower bound of $\Omega\left(\sqrt{\ky}T_{\text{nc}}\right) = \Omega\left(\sqrt{\ky}/\epsilon^2\right)$. 

However, the requirement that $\fmt$ behaves similarly to $\fnc$ poses another challenge as it basically means that the red arrows in chain (d) dominate the bold blue ones, although their numbers are roughly equal. To overcome this challenge, we carefully design the green arrows directly and the bold blue arrows indirectly. We also need to restrict $n$ such that $n= \cO(\sqrt{\ky})$.

\paragraph{Construction of the hard instance.}
The formal expression of $\ft:(\R^{T}\times\R^{T-1})\times \R^{n(T-1)}\to \R$ is given by 
\begin{align}
	\ft(\vx,\vz;\vby) = -\Psi(1)\Phi(x_1) &+ \sum_{i=2}^{T}\left[\Psi(-z_i)\Phi(-x_i)-\Psi(z_i)\Phi(x_i)\right]\nonumber\\
	&\quad +\sum_{i=1}^{T-1}\gd(x_i, z_{i+1};\vby^{(i)})+\sum_{i=1}^{T-1}\left[c_1 x_i^2 + c_2z_{i+1}^2\right],
	\label{equ:unscaled_instance}
\end{align}
where $\vx,\vz$ are the minimization variables and $\vby=[\vby^{(1)},\ldots,\vby^{(T-1)}]$, where each $\vby^{(i)}\in\R^n$, is the maximization variable. The last term of~\eqref{equ:unscaled_instance} indirectly affects the bold blue arrows in Figure~\ref{fig:chain_plot}(d), where $c_1$ and $c_2$ are some parameters (not necessarily nonnegative) we choose later to obtain a much simpler expression for bold blue arrows to analyze. The function $\gd: (\R\times \R)\times \R^n\to \R$ is defined as
\begin{align}
 \gd(x,z;\vy) := &-\frac{1}{2}\sum_{i=1}^{n-1} (y_i-y_{i+1})^2-\frac{1}{2n^2}\norm{\vy}_2^2+\sqrt{\frac{C}{n}}\left(xy_1-\half zy_n\right)\nonumber
 \\
 =&-\frac{1}{2}\vy^\top \left(\frac{1}{n^2} I_n+A\right)\vy+ \sqrt{\frac{C}{n}}\vb_{x,z}^\top \vy,
\end{align} 
where $C>0$ is a large enough numerical constant specified later and $\vb_{x,z}:= x\ve_1-\half z\ve_n$. Here the term $\sqrt{\frac{C}{n}}\vb_{x,z}^\top \vy$ characterizes the green arrows in Figure~\ref{fig:chain_plot}, where the $\sqrt{\frac{C}{n}}$ factor ensures $\gd_m(x,z):= \max_{\vy\in\R^n} \gd(x,z;\vy)$ has an $\cO(1)$ dependence on $n$ as shown in Lemma~\ref{lem:quadratic_max}. The matrix $A\in\R^{n\times n}$ is defined as
\begin{align}
	A:=\begin{pmatrix}
	1&-1&  \\
	-1&2&-1\\
	&-1&\ddots&\ddots\\
	&&\ddots&2&-1\\
	&&&-1&1
	\end{pmatrix}, 
	\label{equ:def_A}
\end{align}
which is the finite version of~\eqref{equ:def_A_nest} in \cite{nesterov2018lectures}'s instance except that its $(1,1)$-th and $(n,n)$-th entries are $1$. This change is necessary in our example.

Note that for fixed scalar variables $x$ and $z$, $\gd(x,z;\cdot)$ is $\frac{1}{n^2}$-strongly-concave. Also, $\ft$ is $\Omega(1)$-smooth (Lemma~\ref{lem:bd_smt}). Therefore, $\ky=\Omega(n^2)$, implying we should choose $n=\cO(\sqrt{\ky})$.

We can compute
\begin{align}
	\gd_m(x,z):= \max_{\vy\in\R^n} \gd(x,z;\vy)=\frac{C}{2n}\vb_{x,z}^\top \left(\frac{1}{n^2} I_n+A\right)^{-1} \vb_{x,z}.
	\label{equ:gm}
\end{align}
The following lemma shows that $\gd_m$ actually has a much simpler expression.

\begin{lemma}
\label{lem:quadratic_max}
	Suppose $n\ge10$, we have
	 \begin{align}
		\gd_m(x,z)=C\left(\frac{a_1}{2}x^2-\half a_2xz+\frac{a_1}{8} z^2\right),
	\end{align}
	where $a_1,a_2>0$ are "almost" numerical constants. That is to say, although $a_1,a_2$ depends on $n$, we have $0<d_1\le a_1\le f_1$ and $0<d_2\le  a_2\le f_2$ and $d_1,d_2,f_1,f_2$ are numerical constants.
\end{lemma}

Choosing $c_1=C(a_2-a_1)/2$, $c_2=C(a_2-a_1)/8$ and $C=12/a_2$, we obtain
\begin{align*}
	\fmt(\vx,\vz) :=& \max_{\vby\in\R^{n(T-1)}}\ft(\vx,\vz;\vby)\\
	=&-\Psi(1)\Phi(x_1)+\sum_{i=2}^{T}\left[\Psi(-z_i)\Phi(-x_i)-\Psi(z_i)\Phi(x_i)\right]+6\sum_{i=1}^{T-1}\left(x_i-\half z_{i+1}\right)^2.
\end{align*}

We show in the appendix that $\fmt$ shares similar properties to $\fnc$. Then we can prove Theorem~\ref{thm:lb-nsc} by appropriately rescaling $\fmt$ in the same way as done in \citep{carmon2019lower}. The detailed proof is deferred in Appendix~\ref{sec:appendix_deterministic}.

\subsection{Construction of the hard instance in the stochastic setting}

We start this section by discussing why the techniques in \citep{arjevani2019lower} do not directly apply to our construction in the deterministic setting. The following lemma from \citep{arjevani2019lower} shows how to construct a probability-$p$ zero-chain by constructing a stochastic first-order oracle over a given zero-chain.
\begin{lemma}[{\citep[Lemma 3]{arjevani2019lower}}]
	\label{lem:arj-l3}
	Let $f: \cX \to \R$ be a zero-chain on $\cX \subset \R^T$.
	For $\vx \in \cX$, let $i^{\ast}(\vx) := \inf\{i \in [T]: x_i = 0\}$ be the next coordinate to discover. For $p \in (0, 1]$, define
	\begin{align*}
		[\vg(\vx,\xi)]_i := \left\{
			\begin{array}{ll}
				\frac{\xi}{p} \nabla_i f(\vx) & \text{ if } i = i^{\ast}(\vx) \\
				\nabla_i f(\vx) & \text{otherwise,}
			\end{array}\right.
	\end{align*}
	where $\xi \sim \Ber(p)$. Suppose $\norm{\nabla f(\vx)}_\infty \le G$ for all $\vx \in \cX$. Then $O: \vx \mapsto (f(\vx,\vg(\vx,\xi))$ is a stochastic first-order oracle with bounded variance $\sigma^2\le G^2(1-p)/p $. Also, $f$ equipped with $O$ is a probability-$p$ zero-chain.
\end{lemma}
Therefore, a tempting way to obtain our hard instance in the stochastic setting is to construct a probability-$p$ zero-chain as in Lemma~\ref{lem:arj-l3} directly over $\ft$. However, as suggested by Lemma~\ref{lem:arj-l3}, to ensure the variance of the oracle is bounded, we require $\norm{\nabla \ft(\vx,\vz;\vby)}_\infty\le G$ for some bounded $G<\infty$. Actually we need $G= \cO(1)$ to obtain a nontrivial lower bound. However, $\ft$ has an unconstrained quadratic component whose gradient is unbounded over the whole space. Therefore, we have to restrict its domain to be a bounded hypercube. However, it turns out to be impossible to find any hypercube such that $\norm{\nabla \ft(\vx,\vz;\vby)}_\infty\le \cO(1)$ without losing the properties of $\fmt$. To overcome this difficulty, we need to carefully modify the quadratic components in $\ft$, i.e., the blue and green arrows in Figure~\ref{fig:chain_plot}(c).

Formally, define the hypercube $\cC^d_R\subset\R^d$ as 
\begin{align*}
    \cC^d_R= \{\vx\in\R^d\mid \norm{\vx}_\infty\le R\}.
\end{align*}
Our hard instance in the stochastic setting, $\fts:(\bcX\times \bcZ)\times \bcY\to \R$, where $R_1$ and $R_2$ are postive numerical constants to be set later, is given by 
\begin{align}
	\fts(\vx,\vz;\vby) = -\Psi(1)\Phi(x_1) &+ \sum_{i=2}^{T}\left[\Psi(-z_i)\Phi(-x_i)-\Psi(z_i)\Phi(x_i)\right]\nonumber\\
	&\quad +\sum_{i=1}^{T-1}\gs(x_i, z_{i+1};\vby^{(i)})+\sum_{i=1}^{T-1}\left[c_1 x_i^2 + c_2z_{i+1}^2\right],
	\label{equ:unscaled_instance_sg}
\end{align}
where
\begin{align}
    \gs(x, z;\vy):=&\frac{C}{n}\left[-\frac{1}{2}\vy^\top \left(\frac{1}{n^2} I_n+A\right)\vy+ \vb_{x,z}^\top \vy\right].
    \label{equ:gs}
\end{align}
Here $C,c_1,c_2,\vb_{x,z}$ are the same as those in \eqref{equ:unscaled_instance}. Also, we choose $R_2\ge 30 R_1$ and $R_1\ge2$. Different from $\gd$, $\gs$ is $C/n^3$ stongly concave, which implies we should choose $n=\cO({\ky}^{1/3})$. We will show in Lemma~\ref{lem:eqv_of_gm} that 
\begin{align*}
    \gs_m(x, z):=\max_{\vy\in \cC^{n}_{nR_2} }\gs(x, z;\vy)= \gd_m(x, z),
\end{align*}
where $\gd_m$ is the quadratic function defined in ~\eqref{equ:gm}. 
Therefore for all $\vx\in\bcX$ and $\vz\in\bcZ$,
\begin{align*}
    \fmts(\vx,\vz)=\fmt(\vx,\vz).
\end{align*}
That is, the function after maximizing over $\vby$ has the same expression as before. 

It is also straightforward to verify that $\norm{\nabla \ft(\vx,\vz;\vby)}_\infty\le \cO(1)$ (Lemma~\ref{lem:bd_smt_sg}.\textit{iv}). Then we are able to construct a probability-$p$ zero-chain as in Lemma~\ref{lem:arj-l3} over the rescaled version of $\ft$. According to Lemma~\ref{lem:arj-l3}, to ensure the variance of the rescaled stochastic oracle is bounded by $\sigmav^2$, $p$ has to be at least $\Omega(\epsilon^2/\sigma^2)$. Then we obtain a lower bound for the stochastic setting of $$\Omega\left(\frac{nT_{\text{nc}}}{p}\right)=\Omega\left(\frac{\ky^{1/3}}{\epsilon^4}\right).$$
Note that the deterministic lower bound is $\Omega(\frac{\sqrt{\ky}}{\epsilon^2})$ which is a special case of the stochastic setting. Therefore we derive a lower bound of 
	\begin{align*}
	    \Omega\left(\max\left\{\frac{\sqrt{\ky}}{\epsilon^2},\frac{\ky^{1/3}\sigma^2}{\epsilon^4}\right\}\right)=\Omega\left(\frac{\sqrt{\ky}}{\epsilon^2}+\frac{\ky^{1/3}}{\epsilon^4}\right).
	\end{align*}
Detailed analyses are deferred to Appendix~\ref{sec:appendix_stochastic}.

\section{Conclusion and discussion}

In this paper, we proved lower bounds on both the deterministic and the stochastic oracle complexities of nonconvex-strongly-concave min-max optimization for first-order zero-respecting algorithms. Our lower bound in the deterministic setting matches the existing upper bound \citep{lin2020nearoptimal} up to log factors. However, there is still a gap between our lower bound and the upper bound in \citep{lin2020gradient} in the stochastic setting. How to close this gap is an open question. Apart from this, several other questions are worth consideration. 

First, one immediate next step is to check if the proposed lower bounds hold for arbitrary, potentially randomized algorithms. We believe that the results are likely to hold but may introduce unexpected complications. 
Second, so far we have focused on nonconvex-strongly-concave min-max optimization. However, it remains open what the tight lower bound is in the more general nonconvex-concave min-max optimization. Moreover, what about nonconvex-nonconcave min-max optimization? To answer this question, a good measure of the suboptimality is a prerequisite. 
Last but not least, we only consider first-order oracles. It is also interesting to obtain a lower bound for functions with higher-order smoothness and oracles.

\bibliographystyle{plainnat}
\bibliography{reference}

\begin{thebibliography}{51}
\providecommand{\natexlab}[1]{#1}
\providecommand{\url}[1]{\texttt{#1}}
\expandafter\ifx\csname urlstyle\endcsname\relax
  \providecommand{\doi}[1]{doi: #1}\else
  \providecommand{\doi}{doi: \begingroup \urlstyle{rm}\Url}\fi

\bibitem[Alkousa et~al.(2019)Alkousa, Dvinskikh, Stonyakin, Gasnikov, and
  Kovalev]{alkousa2019accelerated}
Mohammad Alkousa, Darina Dvinskikh, Fedor Stonyakin, Alexander Gasnikov, and
  Dmitry Kovalev.
\newblock Accelerated methods for composite non-bilinear saddle point problem.
\newblock \emph{arXiv preprint arXiv:1906.03620}, 2019.

\bibitem[Arjevani and Shamir(2016)]{arjevani2016iteration}
Yossi Arjevani and Ohad Shamir.
\newblock On the iteration complexity of oblivious first-order optimization
  algorithms.
\newblock In \emph{International Conference on Machine Learning}, pages
  908--916. PMLR, 2016.

\bibitem[Arjevani et~al.(2019)Arjevani, Carmon, Duchi, Foster, Srebro, and
  Woodworth]{arjevani2019lower}
Yossi Arjevani, Yair Carmon, John~C Duchi, Dylan~J Foster, Nathan Srebro, and
  Blake Woodworth.
\newblock Lower bounds for non-convex stochastic optimization.
\newblock \emph{arXiv preprint arXiv:1912.02365}, 2019.

\bibitem[Ben-Tal et~al.(2009)Ben-Tal, El~Ghaoui, and Nemirovski]{ben2009robust}
Aharon Ben-Tal, Laurent El~Ghaoui, and Arkadi Nemirovski.
\newblock \emph{Robust optimization}.
\newblock Princeton university press, 2009.

\bibitem[Carmon et~al.(2019)Carmon, Duchi, Hinder, and
  Sidford]{carmon2019lower}
Yair Carmon, John~C Duchi, Oliver Hinder, and Aaron Sidford.
\newblock Lower bounds for finding stationary points i.
\newblock \emph{Mathematical Programming}, pages 1--50, 2019.

\bibitem[Carmon et~al.(2021)Carmon, Duchi, Hinder, and
  Sidford]{carmon2021lower}
Yair Carmon, John~C Duchi, Oliver Hinder, and Aaron Sidford.
\newblock Lower bounds for finding stationary points ii: first-order methods.
\newblock \emph{Mathematical Programming}, 185\penalty0 (1-2), 2021.

\bibitem[Chambolle and Pock(2011)]{chambolle2011first}
Antonin Chambolle and Thomas Pock.
\newblock A first-order primal-dual algorithm for convex problems with
  applications to imaging.
\newblock \emph{Journal of mathematical imaging and vision}, 40\penalty0
  (1):\penalty0 120--145, 2011.

\bibitem[Daskalakis et~al.(2017)Daskalakis, Ilyas, Syrgkanis, and
  Zeng]{daskalakis2017training}
Constantinos Daskalakis, Andrew Ilyas, Vasilis Syrgkanis, and Haoyang Zeng.
\newblock Training gans with optimism.
\newblock \emph{arXiv preprint arXiv:1711.00141}, 2017.

\bibitem[Daskalakis et~al.(2020)Daskalakis, Skoulakis, and
  Zampetakis]{daskalakis2020complexity}
Constantinos Daskalakis, Stratis Skoulakis, and Manolis Zampetakis.
\newblock The complexity of constrained min-max optimization.
\newblock \emph{arXiv preprint arXiv:2009.09623}, 2020.

\bibitem[Daskalakis et~al.(2021)Daskalakis, Foster, and
  Golowich]{daskalakis2021independent}
Constantinos Daskalakis, Dylan~J Foster, and Noah Golowich.
\newblock Independent policy gradient methods for competitive reinforcement
  learning.
\newblock \emph{arXiv preprint arXiv:2101.04233}, 2021.

\bibitem[Diakonikolas et~al.(2021)Diakonikolas, Daskalakis, and
  Jordan]{diakonikolas2021efficient}
Jelena Diakonikolas, Constantinos Daskalakis, and Michael Jordan.
\newblock Efficient methods for structured nonconvex-nonconcave min-max
  optimization.
\newblock In \emph{International Conference on Artificial Intelligence and
  Statistics}, pages 2746--2754. PMLR, 2021.

\bibitem[El-Mikkawy(2004)]{ElMikkawy2004OnTI}
M.~El-Mikkawy.
\newblock On the inverse of a general tridiagonal matrix.
\newblock \emph{Appl. Math. Comput.}, 150:\penalty0 669--679, 2004.

\bibitem[Fang et~al.(2018)Fang, Li, Lin, and Zhang]{fang2018spider}
Cong Fang, Chris~Junchi Li, Zhouchen Lin, and Tong Zhang.
\newblock Spider: Near-optimal non-convex optimization via stochastic path
  integrated differential estimator, 2018.

\bibitem[Ghadimi et~al.(2013)Ghadimi, Lan, and Zhang]{ghadimi2013minibatch}
Saeed Ghadimi, Guanghui Lan, and Hongchao Zhang.
\newblock Mini-batch stochastic approximation methods for nonconvex stochastic
  composite optimization, 2013.

\bibitem[Gidel et~al.(2018)Gidel, Berard, Vignoud, Vincent, and
  Lacoste-Julien]{gidel2018variational}
Gauthier Gidel, Hugo Berard, Ga{\"e}tan Vignoud, Pascal Vincent, and Simon
  Lacoste-Julien.
\newblock A variational inequality perspective on generative adversarial
  networks.
\newblock \emph{arXiv preprint arXiv:1802.10551}, 2018.

\bibitem[Golowich et~al.(2020)Golowich, Pattathil, Daskalakis, and
  Ozdaglar]{golowich2020last}
Noah Golowich, Sarath Pattathil, Constantinos Daskalakis, and Asuman Ozdaglar.
\newblock Last iterate is slower than averaged iterate in smooth convex-concave
  saddle point problems.
\newblock In \emph{Conference on Learning Theory}, pages 1758--1784. PMLR,
  2020.

\bibitem[Goodfellow et~al.(2014{\natexlab{a}})Goodfellow, Pouget-Abadie, Mirza,
  Xu, Warde-Farley, Ozair, Courville, and Bengio]{goodfellow2014generative}
Ian~J Goodfellow, Jean Pouget-Abadie, Mehdi Mirza, Bing Xu, David Warde-Farley,
  Sherjil Ozair, Aaron Courville, and Yoshua Bengio.
\newblock Generative adversarial networks.
\newblock \emph{arXiv preprint arXiv:1406.2661}, 2014{\natexlab{a}}.

\bibitem[Goodfellow et~al.(2014{\natexlab{b}})Goodfellow, Shlens, and
  Szegedy]{goodfellow2014explaining}
Ian~J Goodfellow, Jonathon Shlens, and Christian Szegedy.
\newblock Explaining and harnessing adversarial examples.
\newblock \emph{arXiv preprint arXiv:1412.6572}, 2014{\natexlab{b}}.

\bibitem[Hsieh et~al.(2019)Hsieh, Iutzeler, Malick, and
  Mertikopoulos]{hsieh2019convergence}
Yu-Guan Hsieh, Franck Iutzeler, J{\'e}r{\^o}me Malick, and Panayotis
  Mertikopoulos.
\newblock On the convergence of single-call stochastic extra-gradient methods.
\newblock \emph{arXiv preprint arXiv:1908.08465}, 2019.

\bibitem[Ibrahim et~al.(2020)Ibrahim, Azizian, Gidel, and
  Mitliagkas]{ibrahim2020linear}
Adam Ibrahim, Wa{\i}ss Azizian, Gauthier Gidel, and Ioannis Mitliagkas.
\newblock Linear lower bounds and conditioning of differentiable games.
\newblock In \emph{International Conference on Machine Learning}, pages
  4583--4593. PMLR, 2020.

\bibitem[Jin et~al.(2020)Jin, Netrapalli, and Jordan]{jin2020local}
Chi Jin, Praneeth Netrapalli, and Michael Jordan.
\newblock What is local optimality in nonconvex-nonconcave minimax
  optimization?
\newblock In \emph{International Conference on Machine Learning}, pages
  4880--4889. PMLR, 2020.

\bibitem[Kong and Monteiro(2019)]{kong2019accelerated}
Weiwei Kong and Renato~DC Monteiro.
\newblock An accelerated inexact proximal point method for solving
  nonconvex-concave min-max problems.
\newblock \emph{arXiv preprint arXiv:1905.13433}, 2019.

\bibitem[Korpelevich(1976)]{korpelevich1976extragradient}
Galina~M Korpelevich.
\newblock The extragradient method for finding saddle points and other
  problems.
\newblock \emph{Matecon}, 12:\penalty0 747--756, 1976.

\bibitem[Lee et~al.(2020)Lee, Luo, Wei, and Zhang]{lee2020linear}
Chung-Wei Lee, Haipeng Luo, Chen-Yu Wei, and Mengxiao Zhang.
\newblock Linear last-iterate convergence for matrix games and stochastic
  games.
\newblock \emph{arXiv preprint arXiv:2006.09517}, 2020.

\bibitem[Lin et~al.(2020{\natexlab{a}})Lin, Jin, and Jordan]{lin2020gradient}
Tianyi Lin, Chi Jin, and Michael Jordan.
\newblock On gradient descent ascent for nonconvex-concave minimax problems.
\newblock In \emph{International Conference on Machine Learning}, pages
  6083--6093. PMLR, 2020{\natexlab{a}}.

\bibitem[Lin et~al.(2020{\natexlab{b}})Lin, Jin, and
  Jordan]{lin2020nearoptimal}
Tianyi Lin, Chi Jin, and Michael.~I. Jordan.
\newblock Near-optimal algorithms for minimax optimization, 2020{\natexlab{b}}.

\bibitem[Lu et~al.(2020)Lu, Tsaknakis, Hong, and Chen]{lu2020hybrid}
Songtao Lu, Ioannis Tsaknakis, Mingyi Hong, and Yongxin Chen.
\newblock Hybrid block successive approximation for one-sided non-convex
  min-max problems: algorithms and applications.
\newblock \emph{IEEE Transactions on Signal Processing}, 68:\penalty0
  3676--3691, 2020.

\bibitem[Madry et~al.(2017)Madry, Makelov, Schmidt, Tsipras, and
  Vladu]{madry2017towards}
Aleksander Madry, Aleksandar Makelov, Ludwig Schmidt, Dimitris Tsipras, and
  Adrian Vladu.
\newblock Towards deep learning models resistant to adversarial attacks.
\newblock \emph{arXiv preprint arXiv:1706.06083}, 2017.

\bibitem[Mangoubi and Vishnoi(2020)]{mangoubi2020greedy}
Oren Mangoubi and Nisheeth~K. Vishnoi.
\newblock Greedy adversarial equilibrium: An efficient alternative to
  nonconvex-nonconcave min-max optimization, 2020.

\bibitem[Mertikopoulos et~al.(2018)Mertikopoulos, Lecouat, Zenati, Foo,
  Chandrasekhar, and Piliouras]{mertikopoulos2018optimistic}
Panayotis Mertikopoulos, Bruno Lecouat, Houssam Zenati, Chuan-Sheng Foo, Vijay
  Chandrasekhar, and Georgios Piliouras.
\newblock Optimistic mirror descent in saddle-point problems: Going the extra
  (gradient) mile.
\newblock \emph{arXiv preprint arXiv:1807.02629}, 2018.

\bibitem[Mokhtari et~al.(2020)Mokhtari, Ozdaglar, and
  Pattathil]{mokhtari2020convergence}
Aryan Mokhtari, Asuman~E Ozdaglar, and Sarath Pattathil.
\newblock Convergence rate of o(1/k) for optimistic gradient and extragradient
  methods in smooth convex-concave saddle point problems.
\newblock \emph{SIAM Journal on Optimization}, 30\penalty0 (4):\penalty0
  3230--3251, 2020.

\bibitem[Nemirovski(2004)]{nemirovski2004prox}
Arkadi Nemirovski.
\newblock Prox-method with rate of convergence o (1/t) for variational
  inequalities with lipschitz continuous monotone operators and smooth
  convex-concave saddle point problems.
\newblock \emph{SIAM Journal on Optimization}, 15\penalty0 (1):\penalty0
  229--251, 2004.

\bibitem[Nemirovski and Yudin(1983)]{nemirovski1983book}
Arkadi.~S. Nemirovski and David.~B. Yudin.
\newblock \emph{Problem Complexity and Method Efficiency in Optimization}.
\newblock Wiley, 1983.

\bibitem[Nemirovsky(1992)]{Nemirovsky1992InformationbasedCO}
A.~S. Nemirovsky.
\newblock Information-based complexity of linear operator equations.
\newblock \emph{J. Complex.}, 8:\penalty0 153--175, 1992.

\bibitem[Nesterov(2005)]{nesterov2005smooth}
Yu~Nesterov.
\newblock Smooth minimization of non-smooth functions.
\newblock \emph{Mathematical programming}, 103\penalty0 (1):\penalty0 127--152,
  2005.

\bibitem[Nesterov(2018)]{nesterov2018lectures}
Yurii Nesterov.
\newblock \emph{Lectures on convex optimization}, volume 137.
\newblock Springer, 2018.

\bibitem[Neumann(1928)]{neumann1928theorie}
J.~v. Neumann.
\newblock Zur theorie der gesellschaftsspiele.
\newblock \emph{Mathematische annalen}, 100\penalty0 (1):\penalty0 295--320,
  1928.

\bibitem[Nouiehed et~al.(2019)Nouiehed, Sanjabi, Huang, Lee, and
  Razaviyayn]{nouiehed2019solving}
Maher Nouiehed, Maziar Sanjabi, Tianjian Huang, Jason~D Lee, and Meisam
  Razaviyayn.
\newblock Solving a class of non-convex min-max games using iterative first
  order methods.
\newblock \emph{arXiv preprint arXiv:1902.08297}, 2019.

\bibitem[Ostrovskii et~al.(2020)Ostrovskii, Lowy, and
  Razaviyayn]{ostrovskii2020efficient}
Dmitrii~M Ostrovskii, Andrew Lowy, and Meisam Razaviyayn.
\newblock Efficient search of first-order nash equilibria in nonconvex-concave
  smooth min-max problems.
\newblock \emph{arXiv preprint arXiv:2002.07919}, 2020.

\bibitem[Ouyang and Xu(2019)]{ouyang2019lower}
Yuyuan Ouyang and Yangyang Xu.
\newblock Lower complexity bounds of first-order methods for convex-concave
  bilinear saddle-point problems.
\newblock \emph{Mathematical Programming}, pages 1--35, 2019.

\bibitem[Rafique et~al.(2018)Rafique, Liu, Lin, and Yang]{rafique2018non}
Hassan Rafique, Mingrui Liu, Qihang Lin, and Tianbao Yang.
\newblock Non-convex min-max optimization: Provable algorithms and applications
  in machine learning.
\newblock \emph{arXiv preprint arXiv:1810.02060}, 2018.

\bibitem[Rakhlin and Sridharan(2013)]{rakhlin2013optimization}
Alexander Rakhlin and Karthik Sridharan.
\newblock Optimization, learning, and games with predictable sequences.
\newblock \emph{arXiv preprint arXiv:1311.1869}, 2013.

\bibitem[Sion(1958)]{sion1958general}
Maurice Sion.
\newblock On general minimax theorems.
\newblock \emph{Pacific Journal of mathematics}, 8\penalty0 (1):\penalty0
  171--176, 1958.

\bibitem[Thekumparampil et~al.(2019)Thekumparampil, Jain, Netrapalli, and
  Oh]{thekumparampil2019efficient}
Kiran~Koshy Thekumparampil, Prateek Jain, Praneeth Netrapalli, and Sewoong Oh.
\newblock Efficient algorithms for smooth minimax optimization.
\newblock \emph{arXiv preprint arXiv:1907.01543}, 2019.

\bibitem[Tseng(1995)]{tseng1995linear}
Paul Tseng.
\newblock On linear convergence of iterative methods for the variational
  inequality problem.
\newblock \emph{Journal of Computational and Applied Mathematics}, 60\penalty0
  (1-2):\penalty0 237--252, 1995.

\bibitem[Wang et~al.(2019)Wang, Zhang, and Ba]{wang2019solving}
Yuanhao Wang, Guodong Zhang, and Jimmy Ba.
\newblock On solving minimax optimization locally: A follow-the-ridge approach.
\newblock \emph{arXiv preprint arXiv:1910.07512}, 2019.

\bibitem[Wei et~al.(2021)Wei, Lee, Zhang, and Luo]{wei2021last}
Chen-Yu Wei, Chung-Wei Lee, Mengxiao Zhang, and Haipeng Luo.
\newblock Last-iterate convergence of decentralized optimistic gradient
  descent/ascent in infinite-horizon competitive markov games.
\newblock \emph{arXiv preprint arXiv:2102.04540}, 2021.

\bibitem[Yadav et~al.(2017)Yadav, Shah, Xu, Jacobs, and
  Goldstein]{yadav2017stabilizing}
Abhay Yadav, Sohil Shah, Zheng Xu, David Jacobs, and Tom Goldstein.
\newblock Stabilizing adversarial nets with prediction methods.
\newblock \emph{arXiv preprint arXiv:1705.07364}, 2017.

\bibitem[Yang et~al.(2020)Yang, Kiyavash, and He]{yang2020global}
Junchi Yang, Negar Kiyavash, and Niao He.
\newblock Global convergence and variance-reduced optimization for a class of
  nonconvex-nonconcave minimax problems.
\newblock \emph{arXiv preprint arXiv:2002.09621}, 2020.

\bibitem[Zhang et~al.(2020)Zhang, Hong, and Zhang]{zhang2020lower}
Junyu Zhang, Mingyi Hong, and Shuzhong Zhang.
\newblock On lower iteration complexity bounds for the saddle point problems,
  2020.

\bibitem[Zhang et~al.(2021)Zhang, Yang, Guzm{\'a}n, Kiyavash, and
  He]{zhang2021complexity}
Siqi Zhang, Junchi Yang, Crist{\'o}bal Guzm{\'a}n, Negar Kiyavash, and Niao He.
\newblock The complexity of nonconvex-strongly-concave minimax optimization.
\newblock \emph{arXiv preprint arXiv:2103.15888}, 2021.

\end{thebibliography}

\clearpage
\appendix
\section{Useful lemma}
We first present a lemma useful for analyzing the quadratic components in our examples.
\begin{lemma}
    Denote $\alpha=\frac{1}{n^2}$ and let $B=\left(\alpha I_n+A\right)^{-1}$ where $A$ is the matrix defined in~\eqref{equ:def_A}. If $n\ge10,$ we have for all $1\le i\le n$,
    \begin{align*}
        0.1n\le B_{i,1}\le 20 n.
    \end{align*}
    \label{lem:matrix_inverse}
\end{lemma}
\begin{proof}[Proof of Lemma~\ref{lem:matrix_inverse}]
    Let $M$ be the cofactor matrix of $\alpha I_n+A$. We have
	\begin{align*}
		B=\frac{M^\top}{\det\left(\alpha I_n+A\right)}.
	\end{align*}
	So we only need to compute $\det\left(\alpha I_n+A\right)$ and $M_{1,i}$ for all $1\le i\le n$. Note that all of them are determinants of tridiagonal matrices which can be computed using a three-term recurrence relation \citep{ElMikkawy2004OnTI}. Let
	\begin{align*}
		p=1+\frac{\alpha}{2}+\sqrt{\alpha+\frac{\alpha^2}{4}},\quad q=1+\frac{\alpha}{2}-\sqrt{\alpha+\frac{\alpha^2}{4}}
	\end{align*}
	be the solutions of the following equation
	\begin{align*}
		x^2-(2+\alpha)x+1=0.
	\end{align*}
	 By standard calculations, we have
	\begin{align*}
		\det(\alpha I_n+A)&=\frac{\left(\alpha+\frac{\alpha^2}{2}\right)\left(p^{n-1}-q^{n-1}\right)+\alpha\sqrt{\alpha+\frac{\alpha^2}{4}}\left(p^{n-1}+q^{n-1}\right) }{2\sqrt{\alpha+\frac{\alpha^2}{4}}},\\
		M_{1,i}&=\frac{\frac{\alpha}{2}\left(p^{n-i}-q^{n-i}\right)+\sqrt{\alpha+\frac{\alpha^2}{4}}\left(p^{n-i}+q^{n-i}\right)}{2\sqrt{\alpha+\frac{\alpha^2}{4}}}.
	\end{align*}
	Define $D=p^{n-1}$, $E=D-\frac{1}{D}$, and $F=D+\frac{1}{D}$. We have
	\begin{align*}
	    0\le p^{n-i}-q^{n-i}\le E \text{ and } 2\le p^{n-i}+q^{n-i}\le F.
	\end{align*}
	Therefore
	\begin{align*}
	    	&\det(\alpha I_n+A)=\frac{\left(\alpha+\frac{\alpha^2}{2}\right)E+\alpha\sqrt{\alpha+\frac{\alpha^2}{4}}F }{2\sqrt{\alpha+\frac{\alpha^2}{4}}},\\
		&1\le M_{1,n}= M_{1,i}\le M_{1,1}= \frac{\frac{\alpha}{2}E+\sqrt{\alpha+\frac{\alpha^2}{4}}F}{2\sqrt{\alpha+\frac{\alpha^2}{4}}}.
	\end{align*}
Noting $\alpha=\frac{1}{n^2}$, we have
\begin{align*}
	D=p^{n-1}	=&\left(1+\frac{1}{2n^2}+\frac{1}{n}\sqrt{1+\frac{1}{4n^2}}\right)^{n-1}.
\end{align*}
We can bound $2\le D\le 8$ if $n\ge 10$. Then it is straightforward to upper and lower bound $\det(\alpha I_n+A)$ and $M_{1,i}$ and then obtain the bound of $B_{i,1}$. If $n\ge 10$, we have
	\begin{align*}
		0.1n\le &B_{i,1}\le 20n, \forall 1\le i\le n.
	\end{align*}
\end{proof}

\section{Proofs for deterministic lower bound}
\label{sec:appendix_deterministic}

\begin{proof}[Proof of Lemma~\ref{lem:quadratic_max}]
Let $B=\left(\frac{1}{n^2} I_n+A\right)^{-1}$ where $A$ is the matrix defined in~\eqref{equ:def_A}. By symmetry, we have $B_{1,1}=B_{n,n}$ and $B_{1,n}=B_{n,1}$. Then we have
	\begin{align*}
	    \gd_m(x,z)=\frac{C}{2n}\left(B_{1,1}x^2-B_{n,1}xz+\frac{B_{1,1}}{4}z^2
	    \right).
	\end{align*}
	Let $a_1=B_{1,1}/n$ and $a_2=B_{n,1}/n$. By Lemma~\ref{lem:matrix_inverse} we know $0.1\le a_1,a_2\le 20$ and complete the proof.
\end{proof}

To prove the main theorem, we need several additional lemmas.
The following lemma gives a lower bound of the gradient norm when the algorithm hasn't reached the end of the chain.
\begin{lemma}
	\label{lem:gradient_lower_bound}
	If $|z_i|< 1$ for some $i\le T$, then $\norm{\nabla \fmt(\vx,\vz)}_2>\frac{1}{3}$.
\end{lemma}

\begin{proof}[Proof of Lemma~\ref{lem:gradient_lower_bound}]

    We define $z_1\equiv 1$ for simplicity.
	Since $|z_i|<1$ and $|z_1|\ge 1$, we are able to find some $1<j\le i$ to be the smallest $j$ for which $|z_j|<1$. So we know $|z_{j-1}|\ge 1$. We can compute
	\begin{align*}
		\frac{\partial \fmt(\vx,\vz) }{\partial x_{j-1}} =& -\Psi(-z_{j-1})\Phi'(-x_{j-1})-\Psi(z_{j-1})\Phi'(x_{j-1})+12\left(x_{j-1}-\half z_{j}\right)\\ =: &\; p(x_{j-1},z_{j-1})+12\left(x_{j-1}-\half z_{j}\right),\\
		\frac{\partial \fmt(\vx,\vz) }{\partial z_{j}}=&-\Psi'(-z_{j})\Phi(-x_{j})-\Psi'(z_{j})\Phi(x_{j})-6\left(x_{j-1}-\half z_{j}\right)\\=:&\; q(x_j,z_j)-6\left(x_{j-1}-\half z_{j}\right).
	\end{align*}
	Note that Lemma~\ref{lem:properties_phi_psi}.\textit{iv}  implies for all $2\le i\le T$,
	\begin{align*}
		-5< p(x_{j},z_{j}) <0,\quad -20<q(x_j,z_j)<0.
	\end{align*}
	There are two possible cases
	\begin{enumerate}
		\item If $|x_{j-1}|<1$, we have $p(x_{j-1},z_{j-1})<-1$ by Lemma~\ref{lem:properties_phi_psi}.\textit{ii}. Then 
		\begin{align*}
		 	\frac{\partial \fmt(\vx,\vz) }{\partial x_{j-1}}+2\cdot	\frac{\partial \fmt(\vx,\vz) }{\partial z_{j}}=p(x_{j-1},z_{j-1}) +2q(x_j,z_j)<-1.
		\end{align*}
		Therefore we can bound
		\begin{align*}
			\norm{\nabla \fmt(\vx,\vz)}_2\ge \max\left\{\abs{\frac{\partial \fmt(\vx,\vz) }{\partial x_{j-1}}},\abs{\frac{\partial \fmt(\vx,\vz) }{\partial z_{j}}} \right\}>\frac{1}{3}.
		\end{align*}
	\item Otherwise if $|x_{j-1}|\ge1$, we have $12\abs{x_{j-1}-\half z_j}>6$. Since $\abs{p(x_{j-1},z_{j-1})}<5$, we must have \begin{align*}
			\norm{\nabla \fmt(\vx,\vz)}_2\ge\abs{\frac{\partial \fmt(\vx,\vz) }{\partial x_{j-1}}}>1.
	\end{align*}
	\end{enumerate}
	
\end{proof}

Now we verify the smoothness and boundedness requirements of the function class we consider.
\begin{lemma}
	\label{lem:bd_smt}
	$\ft$ and $\fmt$ satisfy the following.
	\begin{enumerate}[i.]
		\item $\fmt(\0,\0)-\inf_{\vx\in\R^{T},\vz\in\R^{T-1}}\fmt(\vx,\vz)\le 12T$.
		\item $\ft$ is $\ell_0$-smooth for some numerical constant $\ell_0$.
	\end{enumerate}
\end{lemma}

\begin{proof}[Proof of Lemma~\ref{lem:bd_smt}]\hfil
	\begin{enumerate}[i.]
		\item First note that $\fmt(\0,\0)=-\Phi(1)\Phi(0)\le 0$. Also, by Lemma~\ref{lem:properties_phi_psi}.\textit{iv}, we have for all $\vx\in\R^{T},\vz\in\R^{T-1}$,
		\begin{align*}
		    \fmt(\vx,\vz)\ge -\Psi(1)\Phi(x_1)-\sum_{i=2}^{T}\Psi(z_i)\Phi(x_i)\ge -12T.
		\end{align*}
		Therefore $\fmt(\0,\0)-\inf_{\vx\in\R^{T},\vz\in\R^{T-1}}\fmt(\vx,\vz)\le 12T$.
		\item Let $\vv=(\vx,\vz,\vby)$ be the variable of $\ft$. We know $\frac{\partial \ft}{\partial v_i\partial vj}\neq 0$ only if $i=j$ or $v_i$ and $v_j$ are directly connected in the chain shown in Figure~\ref{fig:chain_plot} (c). Therefore the Hessian of $\ft$ is tridiagonal if we rearranging the coordinates of $\vv$ according to the order of the chain. By Lemma~\ref{lem:properties_phi_psi}.\textit{iii} and the expression of $\ft$, it is straightforward to verify that each tridiaognal entry of the Hessian is $O(1)$. Therefore the $\ell_2$ norm of the Hessian is $O(1)$, which means $\ft$ is $O(1)$-smooth. 	\end{enumerate}
 \end{proof}
 
 With all the above properties of $\ft$ and $\fmt$, we are ready to show Theorem~\ref{thm:lb-nsc}.

\begin{proof}[Proof of Theorem~\ref{thm:lb-nsc}]
    As in \cite{carmon2019lower}, we construct the hard instance $\fft$ by appropriately rescaling $\ft$ defined in~\eqref{equ:unscaled_instance},
    \begin{align*}
        \fft(\vx,\vz;\vby)=\frac{L\sigmas^2}{\ell_0}\ft\left(\frac{\vx}{\sigmas},\frac{\vz}{\sigmas};\frac{\vby}{\sigmas}\right),
    \end{align*}
    where $\sigmas>0$ is some parameter to be determined later and $\ell_0$ is the smoothness parameter defined in Lemma~\ref{lem:bd_smt}.\textit{ii}. Note that we can show
    \begin{align*}
        \ffmt(\vx,\vz):=\max_{\vby\in\R^{n(T-1)}}\fft(\vx,\vz;\vby)=\max_{\vu\in\R^{n(T-1)}}\frac{L\sigmas^2}{\ell_0}\ft\left(\frac{\vx}{\sigmas},\frac{\vz}{\sigmas};\vu\right)=\frac{L\sigmas^2}{\ell_0}\fmt\left(\frac{\vx}{\sigmas},\frac{\vz}{\sigmas}\right),
    \end{align*}
    which means the order of maximization and rescaling can be interchanged. After the rescaling, $\fft$ is still a zero-chain. Also, if $z_T=0$ for some $(\vx,\vz;\vby)$, Lemma~\ref{lem:gradient_lower_bound} shows that 
    \begin{align*}
        \norm{\nabla \fmt\left(\frac{\vx}{\sigmas},\frac{\vz}{\sigmas}\right) }_2>\frac{1}{3}.
    \end{align*}
    Therefore
    \begin{align*}
        \norm{\nabla \ffmt\left(\vx,\vz\right) }_2=\frac{L\sigmas}{\ell_0}\norm{\nabla \fmt\left(\frac{\vx}{\sigmas},\frac{\vz}{\sigmas}\right) }_2>\frac{L\sigmas}{3\ell_0}.
    \end{align*}
    Choosing $\sigmas=\frac{3\ell_0\epsilon}{L}$ garautees $\norm{\nabla \ffmt\left(\vx,\vz\right) }_2>\epsilon$.
    
    Now we check $\fft\in\cF(L, \mu, \Delta)$. Note that
    \begin{align*}
        \nabla^2 \fft\left(\vx,\vz;\vby\right) =\frac{L}{\ell_0}\nabla^2 \ft\left(\frac{\vx}{\sigmas},\frac{\vz}{\sigmas};\frac{\vby}{\sigmas}\right).
    \end{align*}
    Therefore we know the smoothness parameter of $\fft$ is $L$ and the strong concavity parameter is $\frac{L}{\ell_0 n^2}$. Therefore we should choose
    \begin{align*}
        n=\left\lfloor\sqrt{\frac{L}{\mu \ell_0}}\right\rfloor
    \end{align*}
    to make $\fft$ $\mu$-strongly concave in $\vby$. 
    
    Then it suffices to verify $\ffmt(\0,\0)-\inf_{\vx,\vz}\ffmt(\vx,\vz)\le \Delta$. By Lemma~\ref{lem:bd_smt},
    \begin{align*}
        \ffmt(\0,\0)-\inf_{\vx,\vz}\ffmt(\vx,\vz)=\frac{L\sigmas^2}{\ell_0}\left(\fmt(\0,\0)-\inf_{\vx,\vz}\fmt(\vx,\vz)\right)\le \frac{12LT\sigmas^2}{\ell_0},
    \end{align*}
    which is less than $\Delta$ if choosing
    \begin{align*}
        T=\left\lfloor{\frac{\ell_0\Delta}{12L\sigmas^2}}\right\rfloor=\left\lfloor{\frac{L\Delta}{108\ell_0\epsilon^2}}\right\rfloor.
	\end{align*}
    Since $z_T^{t}=0$ if $t\le n(T-1)$, we conclude that $\norm{\nabla \ffmt(\vx^t,\vz^t)}_2>\epsilon$ whenever
    \begin{align*}
        t\le n(T-1)=\frac{c_0 L\Delta\sqrt{\kappa}}{\epsilon^2}
    \end{align*}
    for some numerical constant $c_0$.
    \end{proof}

\section{Proofs for stochastic lower bound}
\label{sec:appendix_stochastic}
\begin{lemma}
\label{lem:eqv_of_gm}
    Let $\gs_m(x, z):=\max_{\vy\in \cC^{n}_{nR_2} }\gs(x, z;\vy)$. If $R_2\ge 30 R_1$, for every $x,z$ such that $|x|,|z|\le R_1$, we have
    \begin{align*}
        \gs_m(x, z)=\gd_m(x,z),
    \end{align*}
    where $\gd_m$ is the quadratic function defined in~\eqref{equ:gm}.
\end{lemma}
\begin{proof}[Proof of Lemma~\ref{lem:eqv_of_gm}]
    Note that 
    \begin{align*}
        \max_{\vy\in\R^{n}}\gs(x, z;\vy) = \frac{C}{2n}\vb_{x,z}^\top \left(\frac{1}{n^2} I_n+A\right)^{-1} \vb_{x,z} = \gd_m(x,z).
    \end{align*}
    It suffices to verify that 
    \begin{align*}
        \max_{\vy\in \cC^{n}_{nR_2} }\gs(x, z;\vy)=\max_{\vy\in\R^{n}}\gs(x, z;\vy),
    \end{align*}
    i.e., 
    \begin{align*}
        \vy^\ast(x,z):=\argmax_{\vy\in\R^{n}}\gs(x, z;\vy)\in \cC^{n}_{nR_2}.
    \end{align*}
    We can compute that
    \begin{align*}
        \vy^\ast(x,z) =&  \left(\frac{1}{n^2} I_n+A\right)^{-1} \vb_{x,z}=B\cdot \vb_{x,z},
    \end{align*}
    where $B=\left(\frac{1}{n^2} I_n+A\right)^{-1}$ is the matrix defined in Lemma~\ref{lem:matrix_inverse}.
    Let $y_i^\ast(x,z)$ be the $i$-th coordinate of $\vy^\ast(x,z)$ for some $1\le i\le n$. By symmetry of $B$ and Lemma~\ref{lem:matrix_inverse}, we have
    \begin{align*}
        \abs{y_i^\ast(x,z)} =& \abs{xB_{i,1}-\frac{1}{2}zB_{i,n}}\\
        =&\abs{xB_{i,1}-\frac{1}{2}zB_{n-i,1}}\\
        \le &30nR_1\le nR_2.
    \end{align*}
    Therefore $\vy^\ast(x,z)\in \cC^{n}_{nR_2}$ and we complete the proof.
\end{proof}

Now we analyze the properties of $\fts$ and $\fmts$. 

\begin{lemma}
	\label{lem:bd_smt_sg}
	$\fts$ and $\fmts$ satisfies the following.
	\begin{enumerate}[i.]
		\item $\fmts(\0,\0)-\inf_{\vx\in\bcX,\vz\in\bcZ}\fmts(\vx,\vz)\le 12T$.
		\item $\fts$ is $\ell_0$-smooth for some numerical constant $\ell_0$.
		\item $\fmts$ is $\ell_m$-smooth for some numerical constant $\ell_m\ge 1$.
		\item For all $\vx,\vz,\vby$, $\norm{\nabla \fts(\vx,\vz;\vby)}_\infty \le G$ for some numerical constant $G$. 
	\end{enumerate}
\end{lemma}

\begin{proof}[Proof of Lemma~\ref{lem:bd_smt_sg}]
Note that $\bcX\times\bcZ\subset \R^{T}\times\R^{T-1}$. Then \textit{i} and \textit{ii} are direct corollaries of Lemma~\ref{lem:bd_smt}. We can prove \textit{iii} in the same way as \textit{ii}. It is also straightforward to verify \textit{iv} given Lemma~\ref{lem:properties_phi_psi}.\textit{iii} and \textit{iv} and noting the infinity norms of $\vx$, $\vz$, and $\vby$ are all bounded.
\end{proof}

The lemma below shows we cannot find a good solution unless the end of the chain is reached.
\begin{lemma}
	\label{lem:gradient_lower_bound_sg}
	If $|z_i|< 1$ for some $i\le T$, then $(\vx,\vz)$ is not a $1/3$-stationary point of $\fmts$.
\end{lemma}

\begin{proof}[Proof of Lemma~\ref{lem:gradient_lower_bound_sg}]
	Let $1<j\le i$ to be the smallest $j$ for which $|z_j|<1$. Similar to the proof of Lemma~\ref{lem:gradient_lower_bound}, noting $\fmt=\fmts$, we have
	\begin{align*}
		\frac{\partial \fmts(\vx,\vz) }{\partial x_{j-1}} =& p(x_{j-1},z_{j-1})+12\left(x_{j-1}-\half z_{j}\right),\\
		\frac{\partial \fmts(\vx,\vz) }{\partial z_{j}}=& q(x_j,z_j)-6\left(x_{j-1}-\half z_{j}\right),
	\end{align*}
	where
	\begin{align*}
		-5< p(x_{j-1},z_{j-1})
		<0,\quad -20<q(x_j,z_j)<0.
	\end{align*}
	There are two possible cases
	\begin{enumerate}
		\item If $|x_{j-1}|<1$, we know $p(x_{j-1},z_{j-1})<-1$ by Lemma~\ref{lem:properties_phi_psi}.\textit{ii}. Then 
		\begin{align*}
		 	\frac{\partial \fmts(\vx,\vz) }{\partial x_{j-1}}+2\cdot	\frac{\partial \fmts(\vx,\vz) }{\partial z_{j}}=p(x_{j-1},z_{j-1}) +2q(x_j,z_j)<-1.
		\end{align*}
		Therefore we can bound
		\begin{align*}
		 \max\left\{\abs{\frac{\partial \fmts(\vx,\vz) }{\partial x_{j-1}}},\abs{\frac{\partial \fmts(\vx,\vz) }{\partial z_{j}}} \right\}>\frac{1}{3}.
		\end{align*}
		Suppose $u$ is one of $x_{j-1}$ and $z_j$ such that $\abs{\frac{\partial \fmts(\vx,\vz) }{\partial u}}>1/3$. We also know $|u|<1$. 
		Let $\ell_m$ be the smoothness parameter of $\fmts$ defined in Lemma~\ref{lem:bd_smt_sg}.\textit{iii}.
		Define
		\begin{align}
		    u':=u-\frac{1}{\ell_m}\frac{\partial \fmts(\vx,\vz) }{\partial u}.
		    \label{equ:uprime}
		\end{align}
		\begin{enumerate}[i.]
		    \item If $|u'|\le R_1$, we have \begin{align*}
		        \ell_m \abs{\cP_{\cC^1_{R_1}}(u')-u}=\ell_m \abs{u'-u}=\abs{\frac{\partial \fmts(\vx,\vz) }{\partial u}}>1/3.
		    \end{align*}
		    \item If $|u'|>R_1$, we know that $\abs{\cP_{\cC^1_{R_1}}(u')}=R_1$. Then we have
		    \begin{align*}
		        \ell_m \abs{\cP_{\cC^1_{R_1}}(u')-u}>\ell_m(R_1-1)\ge 1.
		    \end{align*}
		\end{enumerate}
	\item If $x_{j-1}\ge 1$, we have $12({x_{j-1}-\half z_j})>6$. Since $-5<{p(x_{j-1},z_{j-1})}<0$, we must have \begin{align*}
		{\frac{\partial \fmt(\vx,\vz) }{\partial x_{j-1}}}>1.
	\end{align*}
	Similar to case~1, we use $u$ to denote $x_{j-1}$ and define $u'$ as in~\eqref{equ:uprime}. We know $u'<u$. Therefore
    \begin{enumerate}[i.]
		    \item If $|u'|\le R_1$, we have \begin{align*}
		        \ell_m \abs{\cP_{\cC^1_{R_1}}(u')-u}=\abs{\frac{\partial \fmts(\vx,\vz) }{\partial u}}>1.
		    \end{align*}
		    \item If $u'<-R_1$, we know that ${\cP_{\cC^1_{R_1}}(u')}=-R_1$. Then we have
		    \begin{align*}
		        \ell_m \abs{\cP_{\cC^1_{R_1}}(u')-u}>\ell_m(R_1+1)\ge 1.
		    \end{align*}
		\end{enumerate}	
	\item If $x_{j-1}\le -1$, we have we have $12({x_{j-1}-\half z_j})<-6$. Since $-5<{p(x_{j-1},z_{j-1})}<0$, we must have \begin{align*}
		{\frac{\partial \fmt(\vx,\vz) }{\partial x_{j-1}}}<-6<-1.
	\end{align*}
	Then similar to case~2, we can show $\ell_m \abs{\cP_{\cC^1_{R_1}}(u')-u}>\ell_m(R_1+1)\ge 1$.
	\end{enumerate}
To sum up, we have
\begin{align*}
    \ell_m \norm{ \cP_{\bcX\times\bcZ}\left((\vx,\vz)-\frac{1}{\ell_m}\nabla \fmts(\vx,\vz)\right)-(\vx,\vz) }_2\ge \ell_m \abs{\cP_{\cC^1_{R_1}}(u')-u}> 1/3,
\end{align*}
i.e., $(\vx,\vz)$ is not a $1/3$-stationary point of $\fmts$.
	
\end{proof}

With all the lemmas above, we are ready to prove Theorem~\ref{thm:lb-nsc-s}.
\begin{proof}[Proof of Theorem~\ref{thm:lb-nsc-s}]
	Similar to the proof of Theorem~\ref{thm:lb-nsc}, we show the lower bound by appropriately rescaling $\fts$ as well as its domain. 
	Formally, define $\ffts:\left(\cC^{T}_{\sigmas R_1}\times \cC^{T-1}_{\sigmas R_1}\right)\times \cC^{n(T-1)}_{\sigmas n R_2}\to \R$ as
    \begin{align*}
        \ffts(\vx,\vz;\vby) = \frac{L\sigmas^2}{\ell_0} \fts\left(\frac{\vx}{\sigmas},\frac{\vz}{\sigmas};\frac{\vby}{\sigmas}\right),
    \end{align*}
    where $\sigmas>0$ is some parameter to be determined later and $\ell_0$ is the smoothness parameter defined in Lemma~\ref{lem:bd_smt_sg}.\textit{ii}. Note that we can show
    \begin{align*}
        \ffmts(\vx,\vz):=&\max_{\vby\in\cC^{n(T-1)}_{\sigmas n R_2} }\ffts(\vx,\vz;\vby)\\
        =&\frac{L\sigmas^2}{\ell_0}\max_{\vu\in\cC^{n(T-1)}_{ nR_2} }\fts\left(\frac{\vx}{\sigmas},\frac{\vz}{\sigmas};\vu\right)\\
        =&\frac{L\sigmas^2}{\ell_0} \fmts\left(\frac{\vx}{\sigmas},\frac{\vz}{\sigmas}\right)
    \end{align*}
    which means the order of maximization and rescaling can be interchanged. After the rescaling, $\ffts$ is still a zero-chain.
    Note that $\ffmts$ is $\ell_m L/\ell_0$-smooth. When $z_T=0$, by Lemma~\ref{lem:gradient_lower_bound_sg},
    \begin{align*}
        &\frac{\ell_mL}{\ell_0}\norm{\cP_{\cC^{T}_{\sigmas R_1}\times \cC^{T-1}_{\sigmas R_1}}
        \left((\vx,\vz)-\frac{\ell_0}{\ell_m L}\nabla \ffmts(\vx,\vz)
        \right)-(\vx,\vz)
        }_2\\
    =& \frac{\ell_mL}{\ell_0}\norm{\cP_{\cC^{T}_{\sigmas R_1}\times \cC^{T-1}_{\sigmas R_1}}
        \left(\sigmas\left(\frac{\vx}{\sigmas},\frac{\vz}{\sigmas}\right)-\frac{\sigmas}{\ell_m }\nabla \fmts\left(\frac{\vx}{\sigmas},\frac{\vz}{\sigmas}\right)
        \right)-\sigmas\left(\frac{\vx}{\sigmas},\frac{\vz}{\sigmas}\right)
        }_2\\
        =& \frac{L\sigmas}{\ell_0} \ell_m\norm{\cP_{\cC^{T}_{ R_1}\times \cC^{T-1}_{ R_1}}
        \left(\left(\frac{\vx}{\sigmas},\frac{\vz}{\sigmas}\right)-\frac{1}{\ell_m }\nabla \fmts\left(\frac{\vx}{\sigmas},\frac{\vz}{\sigmas}\right)
        \right)-\left(\frac{\vx}{\sigmas},\frac{\vz}{\sigmas}\right)
        }_2\\
        >&\frac{L\sigmas}{3\ell_0}.
    \end{align*}
    Choosing $\sigmas = \frac{6\ell_0\epsilon}{L}$ guarantees such $(\vx, \vz)$ is not a $2\epsilon$-stationary point of $\ffts$.
    
    Now we check $\ffts\in\cF(L, \mu, \Delta)$. Note that
    \begin{align*}
        \nabla^2 \ffts\left(\vx,\vz;\vby\right) =\frac{L}{\ell_0}\nabla^2 \fts\left(\frac{\vx}{\sigmas},\frac{\vz}{\sigmas};\frac{\vby}{\sigmas}\right).
    \end{align*}
    We know the smoothness parameter of $\ffts$ is $L$ and the strong concavity parameter is $\frac{L}{\ell_0 n^3}$. Therefore we should choose
    \begin{align*}
        n=\left\lfloor\left({\frac{L}{\mu \ell_0}}\right)^{1/3}\right\rfloor
    \end{align*}
    to make $\ffts$ $\mu$-strongly concave in $\vby$. Then it suffices to show $\ffmts(\0,\0)-\inf_{\vx,\vz}\ffmts(\vx,\vz)\le \Delta$. By Lemma~\ref{lem:bd_smt},
    \begin{align*}
        \ffmts(\0,\0)-\inf_{\vx,\vz}\ffmts(\vx,\vz)=\frac{L\sigmas^2}{\ell_0}\left(\fmts(\0,\0)-\inf_{\vx,\vz}\fmts\left(\frac{\vx}{\sigmas},\frac{\vz}{\sigmas}\right)\right)\le \frac{12LT\sigmas^2}{\ell_0},
    \end{align*}
    which is no greater than $\Delta$ if choosing
    \begin{align*}
        T=\left\lfloor{\frac{\ell_0\Delta}{12L\sigmas^2}}\right\rfloor=\left\lfloor{\frac{L\Delta}{432\ell_0\epsilon^2}}\right\rfloor.
	\end{align*}

	Now we construct the stochastic gradient oracle in the same way as~\citep{arjevani2019lower}. We perturb the gradient only on the next coordinate to discover, so that we reveal its value with probability $p$. Let $i^{\ast}(\vx, \vz; \vby)$ denote the next coordinate to discover in the zero-chain in Figure~\ref{fig:chain_plot}(c). Precisely, we set the stochastic gradient to be 
	\begin{align*}
		\vg(\vx, \vz; \vby;\xi)_i = \left\{
		\begin{array}{ll}
			\frac{\xi}{p} \nabla_i \ffts(\vx, \vz; \vby) & \text{ if } i = i^{\ast}(\vx, \vz; \vby) \\
			\nabla_i \ffts(\vx, \vz; \vby) & \text{otherwise,}
		\end{array}\right.
	\end{align*}
	where $\xi \sim \Ber(p)$. By Lemma~\ref{lem:arj-l3}, $\ffts$ is a probability-$p$ zero-chain with this oracle which has variance bounded by
	\begin{align*}
		\E\left[\norm{\vg(\vx, \vz; \vby;\xi) - \nabla \ffts(\vx, \vz; \vby)}_2^2\right] \le \left(\frac{G L\sigmas}{\ell_{0}}\right)^2 \frac{1-p}{p} = 36\epsilon^2 G^2 \frac{1-p}{p}.
	\end{align*}
	Hence, the variance is no greater than $\sigmav^2$ if $p = \min\{1, \frac{36 \epsilon^2 G^2}{\sigmav^2}\}$. By Lemma~\ref{lem:arj-l1}, with probability $1 - \delta$, $z_{T}^{t} = 0$ for all 
	\begin{align*}
		t \le \frac{n (T - 1) - \log(1/\delta)}{2p}. 
	\end{align*}
	Then taking $\delta = 1/2$ yields that whenever
	\begin{align*}
		t \le \frac{n (T - 1) - 1}{2 \frac{36 \epsilon^2 G^2}{\sigmav^2}} = \frac{c'nT \sigmav^2}{\epsilon^2 G^2} = \frac{c'_0 L \Delta \sigmav^2 \kappa^{1/3}}{\epsilon^4},
	\end{align*}
	for some constant $c', c_0 > 0$, we have 
	\begin{align*}
		\E\left[\frac{\ell_mL}{\ell_0}\norm{\cP_{\cC^{T}_{\sigmas R_1}\times \cC^{T-1}_{\sigmas R_1}}
        \left((\vx,\vz)-\frac{\ell_0}{\ell_m L}\nabla \ffmts(\vx,\vz)
        \right)-(\vx,\vz)
        }_2\right] \ge \frac{1}{2} \cdot 2\epsilon = \epsilon.
	\end{align*}
	That is, $(\vx^t, \vz^t)$ is not an $\epsilon$-stationary point. So far we have derived a lower bound of $\Omega(\frac{L\Delta\sigma^2\ky^{1/3}}{\epsilon^4})$. Note that the deterministic lower bound is $\Omega(\frac{L\Delta\sqrt{\ky}}{\epsilon^2})$ which is a special case of the stochastic setting. Therefore we derive a lower bound of 
	\begin{align*}
	    \Omega\left(L\Delta\max\left\{\frac{\sqrt{\ky}}{\epsilon^2},\frac{\ky^{1/3}\sigma^2}{\epsilon^4}\right\}\right)=\Omega\left(L\Delta\left(\frac{\sqrt{\ky}}{\epsilon^2}+\frac{\ky^{1/3}\sigma^2}{\epsilon^4}\right)\right).
	\end{align*}
\end{proof}

\end{document}